\documentclass{amsart}
\usepackage{amsmath, amsthm, amssymb}
\usepackage{graphics}
\usepackage{pinlabel}

\DeclareMathOperator{\Mod}{Mod}
\DeclareMathOperator{\Cantor}{Cantor}

\begin{document}

\newtheorem{theorem}{Theorem}[subsection]
\newtheorem{lemma}[theorem]{Lemma}
\newtheorem{corollary}[theorem]{Corollary}
\newtheorem{conjecture}[theorem]{Conjecture}
\newtheorem{proposition}[theorem]{Proposition}
\newtheorem*{proposition*}{Proposition}
\newtheorem{question}[theorem]{Question}
\newtheorem*{answer}{Answer}
\newtheorem{problem}[theorem]{Problem}
\newtheorem*{simplex_theorem}{Superideal Simplex Theorem}
\newtheorem*{claim}{Claim}
\newtheorem*{criterion}{Criterion}
\theoremstyle{definition}
\newtheorem{definition}[theorem]{Definition}
\newtheorem{construction}[theorem]{Construction}
\newtheorem{notation}[theorem]{Notation}
\newtheorem{convention}[theorem]{Convention}
\newtheorem*{warning}{Warning}
\newtheorem*{assumption}{Simplifying Assumptions}
\newtheorem{ju}[theorem]{\textcolor{blue}{Comment (Juliette)}}
\newtheorem{al}[theorem]{\textcolor{blue}{Comment (Alden)}}

\theoremstyle{remark}
\newtheorem{remark}[theorem]{Remark}
\newtheorem{example}[theorem]{Example}
\newtheorem{scholium}[theorem]{Scholium}
\newtheorem*{case}{Case}

\def\Id{\text{Id}}
\def\H{\mathbb H}
\def\Z{\mathbb Z}
\def\N{\mathbb N}
\def\R{\mathbb R}
\def\C{\mathbb C}
\def\CP{{\mathbb {CP}}}
\def\CC{\mathcal C}
\def\HC{\mathcal H}
\def\S{\mathcal S}
\def\Sph{\mathbb S}
\def\P{\mathcal P}
\def\Q{\mathbb Q}
\def\L{\mathcal L}
\def\A{\mathcal A}
\def\E{\mathcal E}
\def\homeo{\textnormal{Homeo}}
\def\inte{\textnormal{int}}
\def\scl{\textnormal{scl}}
\def\loc{\textnormal{loc}}
\def\RG{{\mathcal{R}'}}
\def\RGC{\mathcal{R}}

\newcommand{\marginal}[1]{\marginpar{\tiny #1}}

\title{The Gromov boundary of the ray graph}
\author{Juliette Bavard}
\address{Institut de Math\'ematiques de Jussieu-Paris Rive Gauche, UPMC}
\email{juliette.bavard@imj-prg.fr}
\author{Alden Walker}
\address{Center for Communications Research \\ La Jolla, CA 92121}
\email{akwalke@ccrwest.org}

\begin{abstract}
The ray graph is a Gromov-hyperbolic graph on which the mapping class group of the plane minus a Cantor set acts by isometries. We give a description of the Gromov boundary of the ray graph in terms of cliques of long rays on the plane minus a Cantor set.  As a consequence, we prove that the Gromov boundary of the ray graph is homeomorphic to a quotient of a subset of the circle.
\end{abstract}

\maketitle

\setcounter{tocdepth}{1}
\tableofcontents

\section{Introduction}

\subsection{Background}

Let $S = \R^2-\Cantor$ be the plane minus a Cantor set.  The mapping class group $\Mod(S)$ arises naturally in dynamics problems of group actions on the plane and through \emph{artinization} of groups acting by homeomorphisms on Cantor sets (see for example \cite{Calegari-blog}, \cite{Funar-K-S}).  The curve graph and curve complex are important tools in the study of mapping class groups of surfaces of finite type.  However, the standard definition of the curve graph for $S$ yields a graph with finite diameter since any two closed curves are disjoint from a sufficiently small loop.  Therefore, we must find a more suitable object of study which usefully captures the action of $\Mod(S)$.

Such an object, the \emph{ray graph}, was defined by Danny Calegari in \cite{Calegari-blog}.  The mapping class group $\Mod(S)$ naturally acts on the ray graph, and the first author proved in~\cite{Juliette} that it has infinite diameter and is Gromov-hyperbolic. Hence, this graph can play a similar role to curve complexes for finite type surfaces.

Erica Klarreich gives in \cite{Klarreich} a description of the Gromov boundary of curve complexes of finite topological type surfaces in terms of minimal filling laminations on the surface (see also \cite{Pho-On} for a more recent proof using infinite unicorn paths). In the case of $S$, complications arise in generalizing this definition as the surface is of infinite type and in particular its lamination space is non compact in the Hausdorff topology. 

Our goal in this paper is to give a simple characterization of the Gromov boundary of the ray graph.

\subsection{Results}

We give a brief overview of our results, leaving precise definitions for Section~\ref{section:completed_ray_graph}.  A \emph{short ray} is a simple arc embedded in $S$ up to isotopy connecting infinity to a point on the Cantor set.  The \emph{ray graph} is the graph whose vertices are short rays and whose edges connect pairs of disjoint rays.  The mapping class group $\Mod(S)$ acts by isometries on the ray graph, which, as proved in \cite{Juliette}, is Gromov hyperbolic and has infinite diameter.

A \emph{long ray} is a simple arc embedded in $S$ which limits to infinity on one end but has no limit point on the other (we define this condition precisely below).  A long ray is more properly called a curve, since an arc implies its domain is closed, but inspired as we are by the arc graph for finite type surfaces, we think of long rays as arcs which accidentally fail to ever end.  A \emph{loop} is an embedded simple arc connecting infinity to itself.  We define a \emph{completed ray graph}, whose vertices are short and long rays and loops and whose edges are pairs of disjoint rays and loops.  We prove that this completed ray graph has infinitely many connected components and describe them (Theorem~\ref{theorem:completed_components}). The main one contains the usual ray graph and is quasi-isometric to it. In particular, it has infinite diameter. The others are \emph{cliques}, i.e. complete sub-graphs of diameter $0$ or $1$. We exhibit a natural bijection between the set of such cliques and the Gromov boundary of the main component, which is also the Gromov boundary of the ray graph.  See Theorem~\ref{theorem:boundary_bijection}.  Thus, roughly speaking, the Gromov boundary of the ray graph is identified with a set of long rays.

Coming from another angle, we describe a specific cover of $S$, called the \emph{conical cover}.  This cover has a circle $S^1$ as a natural boundary, which plays a primary role in our description. Indeed, the main interest of this cover is the existence of a natural injection from the set of vertices of the completed ray graph to this circle boundary. This injection allows us to see the union of the vertices of the cliques as a subset $\mathcal{H}$ of $S^1$. In particular, $\mathcal{H}$ inherits the induced topology of the circle~$S^1$. Quotienting $\mathcal{H}$ by the equivalence relation ``being in the same clique'', we get a quotient space $\mathcal{H} / \sim $  endowed with the quotient topology. We show that the Gromov boundary of the ray graph is homeomorphic to $\mathcal{H} / \sim $.  See Theorem~\ref{theorem:boundary_homeo}.

\subsection{Outline}

In Section~\ref{section:completed_ray_graph}, we define the completed ray graph and the different types of rays that we will encounter.  In Section~\ref{section:unicorn_paths}, we adapt the \emph{infinite unicorn paths} defined by Witsarut Pho-On in \cite{Pho-On} to our situation. Section \ref{section:Gromov_boundaries} reviews Gromov boundaries of non proper Gromov-hyperbolic spaces and provides some necessary technical background relating unicorn paths to quasi-geodesics.  Finally, in Section~\ref{section:bijection}, we describe a bijection between the set of cliques and the Gromov-boundary, and we prove that this bijection is a homeomorphism.

\subsection{Acknowledgements}

We would like to thank Danny Calegari and Fr\'ed\'eric Le Roux for helpful conversations. Juliette Bavard was supported by grants from R\'egion Ile-de-France. Alden Walker was partially supported by NSF grant DMS 1203888.

\section{The completed ray graph}\label{section:completed_ray_graph}

\subsection{The plane minus a Cantor set}

As above, let $S = \R^2 - \Cantor$ be the plane minus a Cantor set.  Note that it does not matter how the removed Cantor set is embedded in the plane, as all embeddings are homeomorphic; see for example \cite{Moise}.  Therefore, we will take the Cantor set to be $K$, the standard middle-thirds Cantor set on the horizontal axis.  In addition, $S$ is homeomorphic to a $2$-sphere minus a Cantor set and a single isolated point $\infty$.  For this reason, we will often abuse notation and think of $S$ as $\Sph^2 - (K \cup\{\infty\})$, where $\Sph^2$ is the Riemann sphere, from which we have removed the above-defined Cantor set $K$ and the point $\infty$.  From this viewpoint, we think of both $K$ and $\infty$ as lying on the equator (real axis) of $\Sph^2$.  The fact that $K$ and $\infty$ divide this equator into infinitely many segments will be important for our purposes.

It will also be useful to fix a complete hyperbolic metric (of the first kind) on $S$.  Figure~\ref{figu:surface_cantor_1} shows the different ways in which we can think about $S$.

\begin{figure}[htb]
\labellist
\pinlabel $\infty$ at 1219 475
\pinlabel $\infty$ at 2144 -25
\endlabellist
\centering
\includegraphics[scale=0.13]{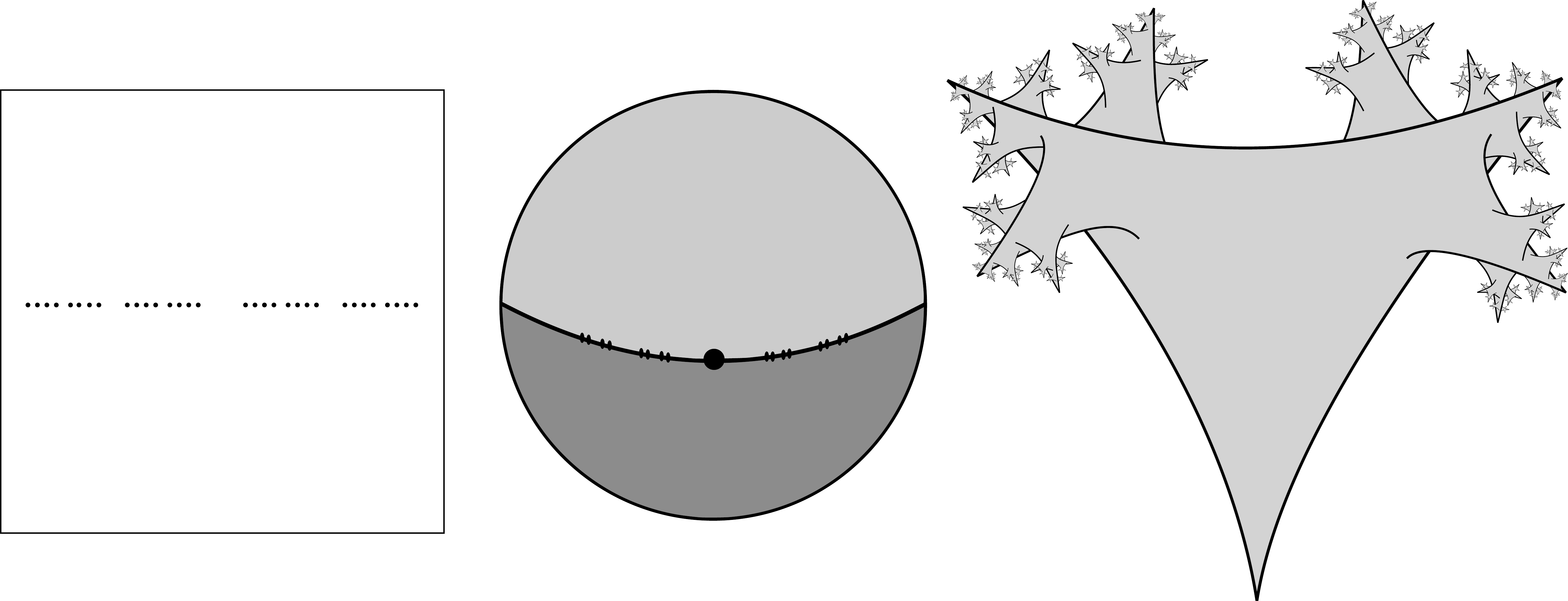}
\vspace{0.3cm}
\caption{Three views of the plane minus a Cantor set: the usual viewpoint, as a sphere minus a Cantor set and a point, and as endowed with a hyperbolic metric}
\label{figu:surface_cantor_1}
\end{figure}

\subsection{The conical cover}

Since $S$ is hyperbolic, we can think of the universal cover of $S$ as the hyperbolic plane.  We won't make particular use of the metric, but it is convenient to know that the boundary of the universal cover is a circle.  We can work directly with the universal cover, but in fact there is a sub-cover which is even more amenable, the \emph{conical cover}, defined as follows.

Think of $S$ as $\Sph^2 - (K \cup\{\infty\})$, and recall that we have chosen an equator in $S$ which contains $\{\infty \} \cup K$. A lift of each hemisphere is a half-fundamental domain (see Figure~\ref{figu:cover}). Choose a lift of $\infty$ on the boundary of the universal cover, and quotient by parabolics around this lift in order to get a cover of $S$ with one special lift $\tilde \infty$ of $\infty$ which bounds only two half fundamental domains, as the cover on the right of Figure~\ref{figu:cover}. Such a cover is said to be a \emph{conical cover}.  We fix a conical cover and denote it by $\tilde{S}$ and its (circle) boundary by $\partial\tilde{S}$.  Call the covering $\pi :\tilde{S} \to S$. See figure \ref{figu:cover2} for two views of a conical cover.

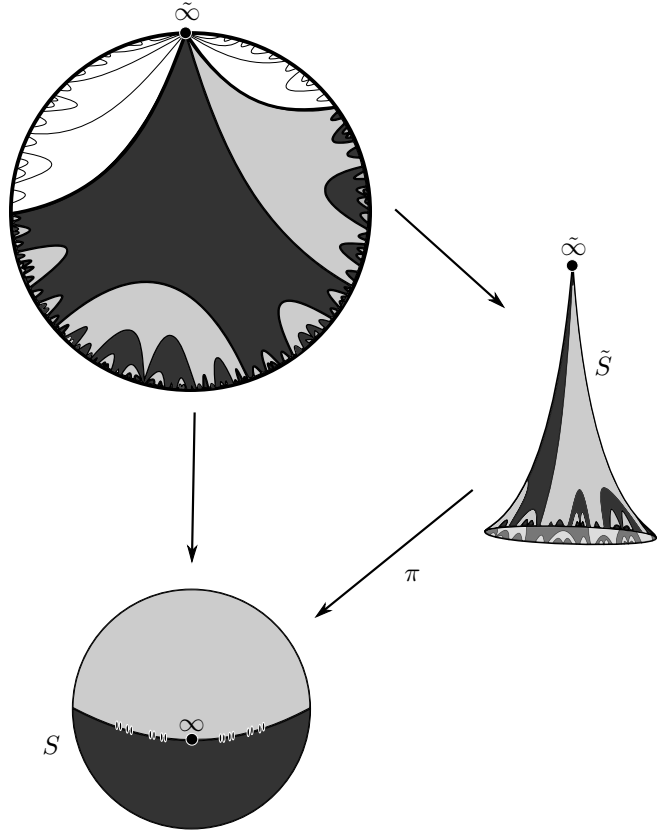
\begin{figure}[htb]
\labellist
\pinlabel $\tilde \infty$ at 84 386 
\pinlabel $\tilde \infty$ at 266 275
\pinlabel $\infty$ at 86 49
\pinlabel $S$ at 20 40
\pinlabel $\tilde S$ at 280 220
\pinlabel $\pi$ at 190 120
\endlabellist
\centering
\vspace{0.2cm}
\includegraphics[scale=0.8]{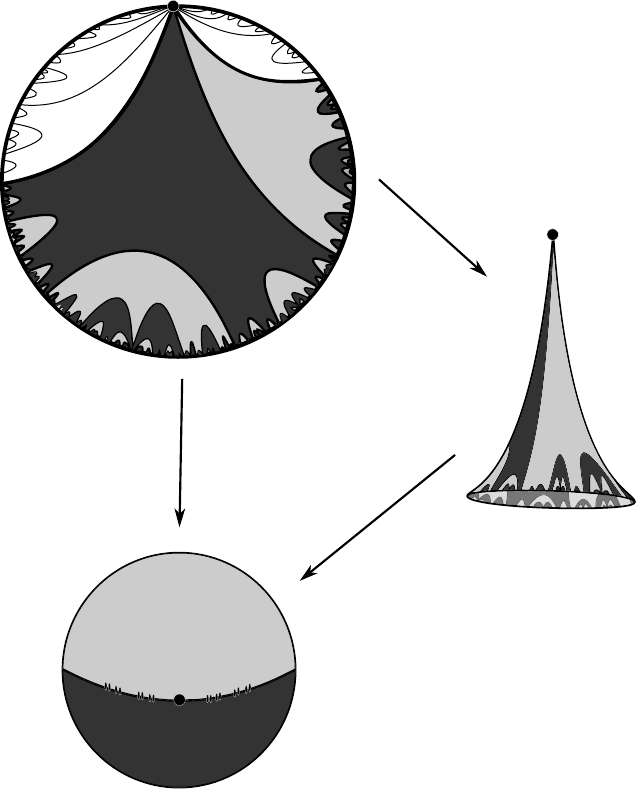}
\caption{The universal cover and a \emph{conical} cover $\tilde S$ of $\Sph^2 -(\{\infty \} \cup K)$}
\label{figu:cover}
\end{figure}

\begin{figure}[htb]
\labellist
\pinlabel $\tilde \infty$ at 42 141
\pinlabel $\tilde \infty$ at 252 65
\endlabellist
\centering
\vspace{0.3cm}
\includegraphics[scale=0.8]{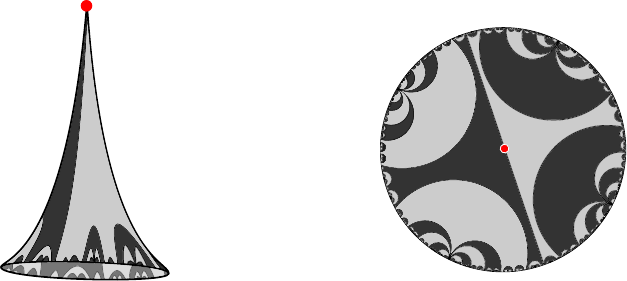}
\caption{Two views of a conical cover: one as a cone, and the other one seen from the top -- we will mainly use the second one}
\label{figu:cover2}
\end{figure}

\subsection{Different types of rays}

\subsubsection{Short rays and loops}

An embedding $\gamma$ of the segment $]0,1[$ in $S$ is said to be:
\begin{itemize}
\item A \emph{loop} if it can be continuously extended in $\Sph^2$ by $\gamma(0)=\gamma(1)=\{\infty\}$.
\item A \emph{short ray} if it can be continuously extended in $\Sph^2$ by $\gamma(0)=\{\infty\}$ and $\gamma(1) \in K$.
\end{itemize}

We are interested in loops and short rays only up to isotopy.  Hereafter, we conflate a short ray or loop with its isotopy equivalence class.  We say that two short rays or loops are disjoint if there are disjoint rays or loops in their equivalence class (i.e. if they can be made disjoint by separate isotopies).  There is a technical issue with loops, which is that we actually consider \emph{oriented} loops.  In practice, this is never an issue, except that we must keep it in mind when we define unicorn paths.  We can take a loop and its reverse to be disjoint.

The conical cover is very useful for understanding short rays and loops because it allows us to make a natural choice of equivalence class representatives, as follows.  Given a representative of a short ray $r$, there is a distinguished lift $\tilde{r}$ from the basepoint $\tilde{\infty}$ in the conical cover.  The ray $\tilde{r}$ limits to some point $p \in \partial\tilde{S}$ (by definition, a lift of the endpoint of $r$ in the Cantor set).  It follows from standard facts about the hyperbolic plane that the equivalence class of $r$ is specified by $p$, and there is a unique geodesic in the equivalence class of $r$ with limit point $p$ which we can choose as the class representative.  Similarly, given a loop, there is a well-defined limit point in $\partial\tilde{S}$ (which is a lift of $\infty$, but not the distinguished lift $\tilde{\infty}$), and we can choose a class representative which is a geodesic in the conical cover.  Note that we conflate this geodesic lift and its image in $S$. Figure~\ref{figu:rays_and_loop} shows examples of a short ray and a loop.

The hyperbolic metric on $\tilde{S}$ is not necessary to decide whether two rays or loops are disjoint.  However, it is very useful because if we choose class representatives which are these geodesics, then these representatives are disjoint if and only if there are any class representatives which are disjoint.  Also, a ray or short loop is simple iff its geodesic representative is simple.  These are standard facts about hyperbolic geodesics.

\subsubsection{Long rays}

We have defined geodesic representatives of short rays and loops.  We observed that short rays and loops are specified by points on the boundary $\partial\tilde{S}$ and that there is a subset of $\partial\tilde{S}$ corresponding to simple short rays and loops.  We could instead start with $\partial\tilde{S}$ and ask: which points $p\in \partial\tilde{S}$ are the endpoints of geodesics whose image under the covering map $\pi:\tilde{S} \to S$ is simple?  We would find the lifts of points in $K$ which are endpoints of simple short rays and the lifts of $\infty$ which are the endpoints of simple loops.  However, we would also find other points on the boundary.  These are the long rays, defined as follows:

\begin{itemize}
\item Let $\tilde\gamma$ be a geodesic ray from $\tilde\infty$ to $\partial\tilde{S}$.  If $\gamma = \pi(\tilde{\gamma})$ is simple and $\gamma$ is not a short ray or loop, then it is a \emph{long ray}.\end{itemize}

Note that again the hyperbolic metric is not necessary for the definition of long rays -- any ray from $\infty$ which lifts to a ray with a well-defined limit point on $\partial\tilde{S}$ is a long ray, but the metric is convenient to pick out a geodesic representative. Figure~\ref{figu:rays_and_loop} gives an example of a long ray.

\begin{figure}[htb]
\centering
\labellist
\pinlabel $\infty$ at 182 433
\pinlabel $\tilde{\infty}$ at 179 148
\pinlabel $p$ at 364 428
\pinlabel $\tilde{p}$ at 278 22

\pinlabel $\infty$ at 560 433
\pinlabel $\tilde{\infty}$ at 557 148
\pinlabel $\tilde{\infty}'$ at 664 224

\pinlabel $\infty$ at 965 433
\pinlabel $\tilde{\infty}$ at 956 146
\endlabellist
\includegraphics[scale=0.3]{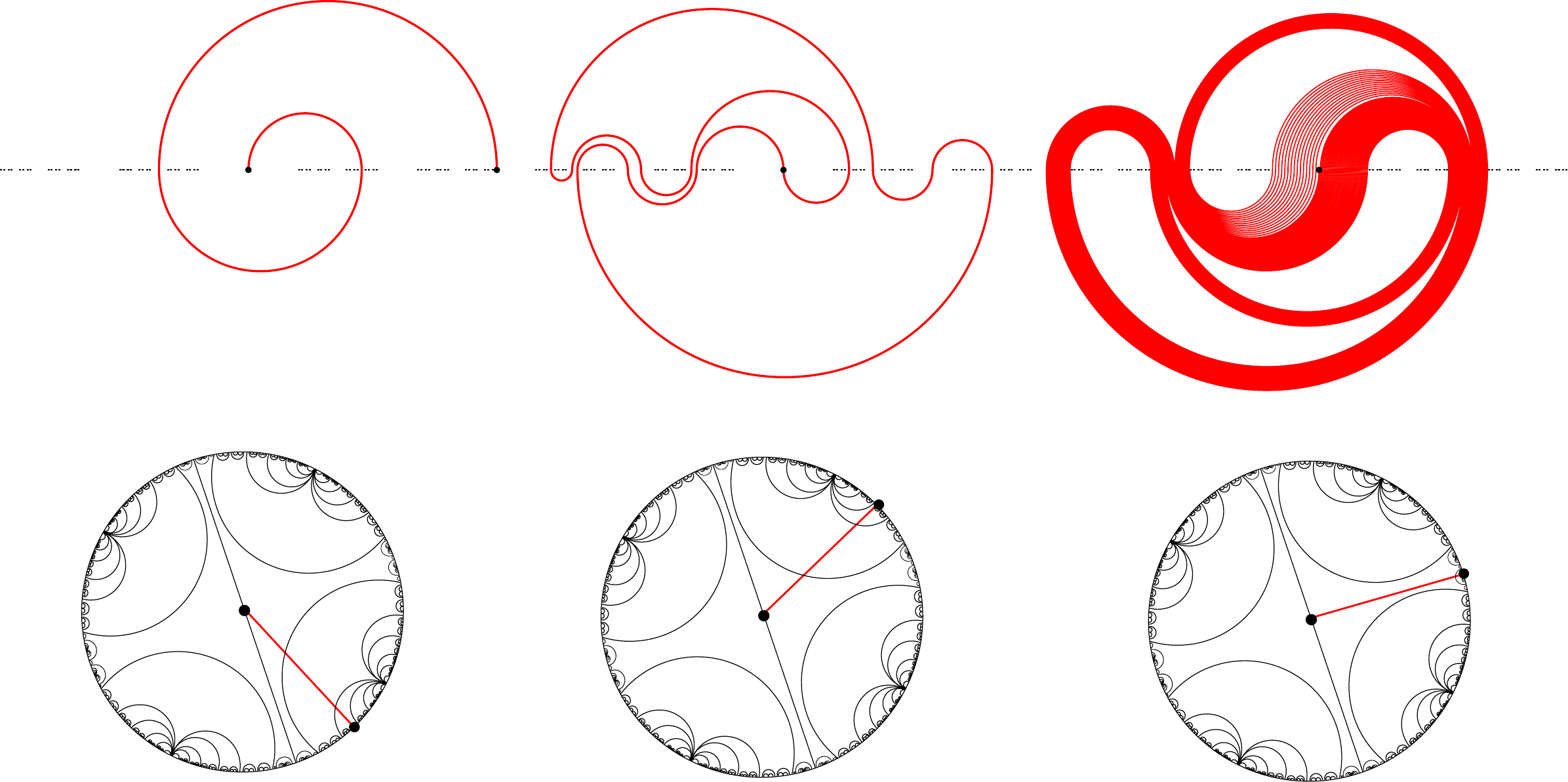}
\vspace{0.2cm}
\caption{A short ray, a loop, and (the beginning of) a long ray.  We show the lifts of the rays and loop in the conical cover.  Note the short ray ends on a lift of a point in the Cantor set, the loop ends on another lift of $\infty$, and the long ray ends on a boundary point which is neither.  Note
the pictures are drawn \emph{on the sphere}, with $\infty$ highlighted in the middle.}
\label{figu:rays_and_loop}
\end{figure}
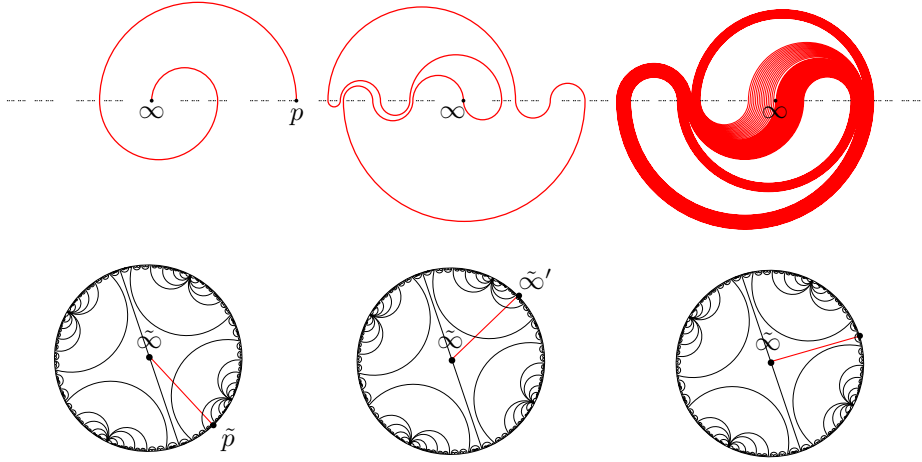

\subsection{Cover-convergence and $k$-beginning}
\label{section:cover_convergence}

The conical cover will be helpful to us especially because it is compact.  We record the idea that rays converge on the boundary.

\begin{definition} We say that a sequence of rays or oriented loops $(x_n)$ \emph{cover-converges} to a geodesic (ray or loop) $l$ on the surface if the sequence of endpoints of the $\tilde{x_n}$ on the boundary of the conical cover converges to a point $p$ such that the image $\pi( (\tilde{\infty} p) )$ of the geodesic $(\tilde{\infty} p)$ by the quotient map of the covering is $l$.
\end{definition}

\begin{lemma}\label{lemma:simple_is_compact}
Let $E \subseteq \partial \tilde{S}$ be the set of endpoints of short rays, loops, and long rays (i.e. the set of all endpoints of geodesics whose projection to $S$ is simple).  Then $E$ is compact.  The endpoints of loops are isolated (are not accumulation points of $E$).
\end{lemma}
\begin{proof}
Since $\partial \tilde{S}$ is compact, to show that $E$ is compact it suffices to show that $E$ is closed.  Let $(\gamma_i)$ be a sequence of geodesic rays in $\tilde{S}$ such that $\pi(\gamma_i)$ is simple (each is a short ray, loop, or long ray).  Suppose that $\gamma_i$ cover-converges to $\gamma$.  To prove the lemma, it suffices to show that $\pi(\gamma)$ is simple.  Suppose not.  Then $\pi(\gamma)$ must have a self-intersection \emph{not} at $\infty$ nor at a point in the Cantor set.  Therefore, $\gamma$ must intersect another lift of $\pi(\gamma)$, and this intersection must occur in the interior of $\tilde{S}$.  Since $\gamma_i$ cover-converges to $\gamma$, we must have $\gamma_i$ converging to $\gamma$ in the Hausdorff metric when $\gamma_i$ and $\gamma$ are thought of as subsets of the disk (analogous to the Poincar\'{e} disk).  Therefore, for $i$ sufficiently large, we must find a similar self-intersection between $\gamma_i$ and another lift of $\pi(\gamma_i)$.  This is a contradiction since all the $\gamma_i$ are simple.

To show that loops are isolated, observe that no simple geodesic ray can cross the equator
adjacent to $\infty$ more than twice (a simple geodesic cannot wrap around $\infty$).
This produces a region about each lift of $\infty$ on $\partial\tilde{S}$ into
which no simple geodesic can enter (except to terminate exactly on that lift of $\infty$).
\end{proof}

\begin{lemma}\label{lemma:cover_converge_limit}
Let $(x_i)$ and $(y_i)$ be sequences of short rays, loops, or long rays such that the pair $x_i$, $y_i$ is disjoint for all $i$.  Suppose that $(x_i)$ cover-converges to $x$ and $(y_i)$ cover-converges to $y$.  Then $x$ and $y$ are disjoint.
\end{lemma}
\begin{proof}
By Lemma~\ref{lemma:simple_is_compact}, we know that each of $x$, $y$ is a simple short ray, loop, or long ray, so the only question is whether they are disjoint.  Suppose they are not.  Then thinking of $x$ and $y$ as their representatives in the conical cover $\tilde{S}$, we see they must have a point of intersection not at the boundary.  But by taking $i$ large enough, we can make $x_i$ and $y_i$ have geodesic representatives in $\tilde{S}$ close enough to $x$ and $y$ that they must intersect.  This is a contradiction.
\end{proof}

Note that we have no information about whether $x$ and $y$ are short rays, loops, or long rays, but whatever they are, they are simple and disjoint.

It is also helpful to have a topological or combinatorial notion of cover-convergence.  This is the notion of \emph{$k$-beginning-like}.  

\begin{definition}\label{def:k_begin}
Let us think of $S$ as $\Sph^2 -(\{\infty \} \cup K)$, and recall that $K$ and $\infty$ divide the equator into infinitely many segments.  We say that two oriented rays or loops \emph{$k$-begin} like each other if they cross the same initial $k$ equatorial segments in the same direction.  Note that a loop or long ray is specified uniquely by the sequence of equatorial segments that it crosses, and note that it is important to record in which direction we cross the segments (or equivalently in which hemisphere the ray starts).
\end{definition}

The definition applies even to short rays, but these might not be specified just by their equator crossings.  Figure~\ref{figu:k_begin} gives an example of two rays which $k$-begin like each other.

\begin{figure}[htb]
\labellist
\pinlabel $\infty$ at 184 57
\endlabellist
\centering
\includegraphics[scale=0.55]{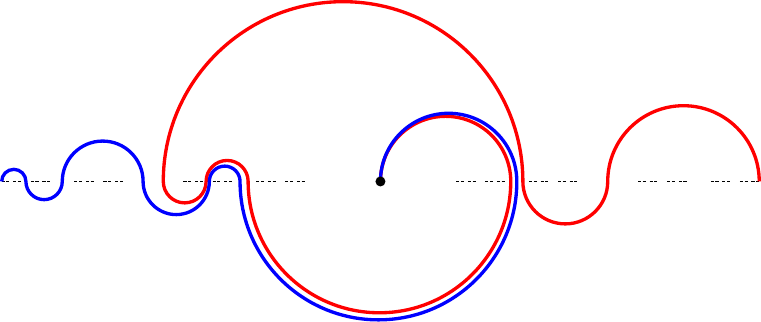}
\caption{Two rays which $4$-begin like each other.}
\label{figu:k_begin}
\end{figure}

\begin{lemma}\label{lemma:cover_converge_iff_k_begin}
Let $(x_i)$ be a sequence of rays or loops and let $x$ be a long ray.  Then $(x_i)$ cover-converges to $x$ if and only if for all $k$ there is an $I$ so that for all $i \ge I$, we have $x_i$ $k$-begins like $x$.
\end{lemma}
\begin{proof}
Because $x$ is a long ray, its endpoint in the boundary of the conical cover is not on the boundary of a fundamental domain of $S$.  Hence cover convergence is equivalent to crossing the same fundamental domains as $x$; that is, $k$-beginning like.
\end{proof}

\subsection{The ray graphs} 

The \emph{ray graph} $\RG$ is the graph whose vertices are short rays up to isotopy and whose edges are pairs of disjoint short rays.  Assigning each edge a length of~$1$ makes $\RG$ into a metric space.  See Figure~\ref{figu:RG_example}.  Since homeomorphisms preserve the property of being disjoint, the mapping class group $\Mod(S)$ acts by isometries on $\RG$.  In \cite{Juliette}, the first author proved that $\RG$ is infinite diameter and hyperbolic.  The analogous \emph{loop graph} defined for loops is, in fact, quasi-isometric to $\RG$.  This relationship is critical in \cite{Juliette}.  We can define yet another graph, the \emph{short-ray-and-loop graph}, whose vertices are both short rays and loops and whose edges join pairs of objects (rays or loops) which are disjoint.

The \emph{completed ray graph} $\RGC$ is the graph whose vertices are short rays, long rays, and loops.  The edges of $\RGC$ are pairs of vertices which can be realized disjointly.  It will be convenient to work in the context of $\RGC$ because we will go back and forth between rays and loops: rays are easier to reason about, and they give a nice description of the Gromov boundary of $\RGC$, but loops are easier to work with than short rays.  Note that $\RG$ is a subgraph of $\RGC$.

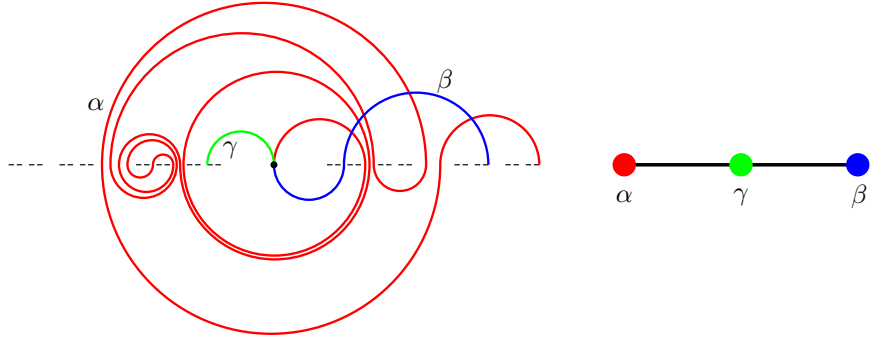
\begin{figure}[htb]
\labellist
\pinlabel $\alpha$ at 60 158
\pinlabel $\beta$ at 299 174
\pinlabel $\gamma$ at 152 127

\pinlabel $\alpha$ at 422 95
\pinlabel $\beta$ at 583 95
\pinlabel $\gamma$ at 502 95
\endlabellist
\centering
\includegraphics[scale=0.55]{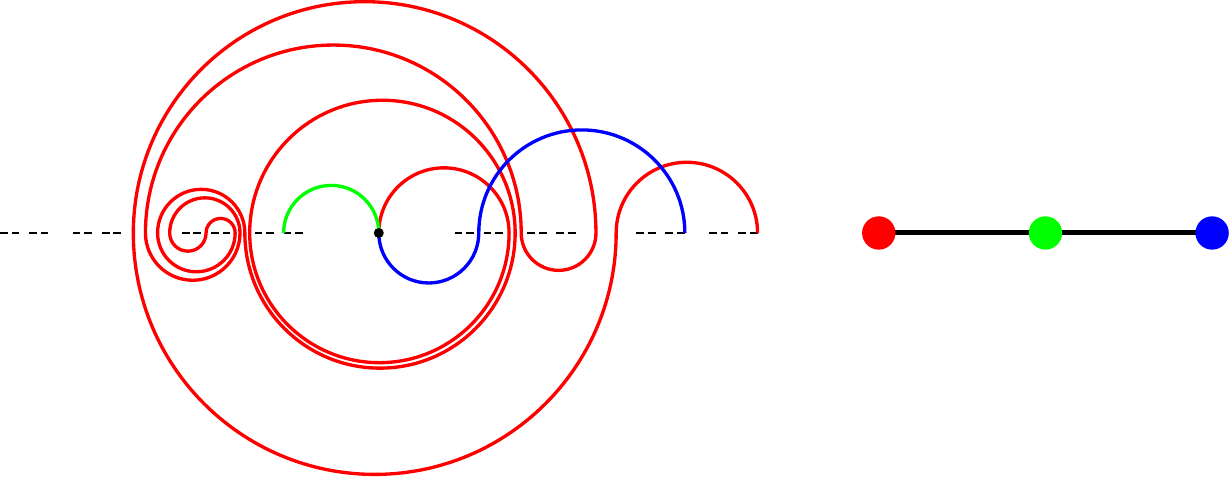}
\caption{Example of three short rays giving a path of length $2$ in the ray graph $\RG$, 
and the corresponding picture in $\RG$.}
\label{figu:RG_example}
\end{figure}

\subsection{Actions}

Although our overarching goal is to understand the mapping class group $\Mod(S)$ of the plane minus a Cantor set, we will not actually be acting by any homeomorphisms in this paper.  Here we are concerned with understanding the boundary of the ray graph $\RG$ so that later we may leverage this understanding to study $\Mod(S)$.  However, it is important to note that $\Mod(S)$ does, in fact, act on $\RGC$.

\begin{lemma}\label{lemma:mod_acts}
The mapping class group $\Mod(S)$ acts by homeomorphisms on the boundary $\partial \tilde{S}$ of the conical cover.  It acts by isometries on $\RG$ and $\RGC$.
\end{lemma}
\begin{proof}
First, we observe that $\Mod(S)$ acts by isometries on $\RG$.  This is obvious from the definition of short rays and loops.  Next, as we require that any mapping class $\phi \in \Mod(S)$ fixes $\infty$, there is a well-defined lift $\tilde{\phi}:\tilde{S} \to \tilde{S}$ fixing $\tilde{\infty}$.  Now, endpoints of short rays and loops (not necessarily simple) are dense in the circle $\partial\tilde{S}$.  As any mapping class must preserve the cyclic order of geodesic rays from $\tilde{\infty}$, the mapping class must preserve the cyclic order of this dense set of ray and loop endpoints and thus extend to a homeomorphism on the boundary $\partial\tilde{S}$.  To avoid issues of isotopy, our definition of a long ray is in terms of endpoints on $\partial\tilde{S}$.  Hence only now, after establishing that $\tilde{\phi}$ induces a homeomorphism on $\partial\tilde{S}$, can we observe that this implies that $\phi$ acts on the set of long rays.  Furthermore, this action preserves disjointness between long rays and short rays and loops (because this can be defined in terms of cyclic orders on the boundary $\partial\tilde{S}$).  Thus $\phi$ induces an isometry on $\RGC$.
\end{proof}

\begin{lemma}\label{lemma:mod_transitive_on_short_rays}
The mapping class group $\Mod(S)$ acts transitively on short rays.
\end{lemma}
\begin{proof}
Let us be given two short rays $a,b$.  We will think of them as being embedded in $S^2 - (\{\infty\} \cup K)$.  The endpoints $p_a, p_b$ of $a,b$ are points in the Cantor set.  By a theorem of Schoenflies (see for example~\cite{Le Roux}, Theorem A.3), there is a homeomorphism $f$ of the sphere such that $f(\infty) = \infty$ and $f(a) = b$ (and hence $f(p_a) = p_b$).  If we had $f(K) = K$, we would be done, but this is probably not the case.  We will find a homeomorphism $g$ of the sphere which fixes $\infty$, $b$, and $p_b$ such that $g(f(K)) = K$.  The composition $g \circ f$ will then be the desired element of $\Mod(S)$.  We let $U = S^2 - \{\infty,p_b\}\cup b$.  Note $U$ is an open disk whose closure is the ray $b$.  Every point of $f(K)$ and $K$ is contained in $U$ except for the point which is labeled $p_b \in K$.  To alleviate confusion, we will denote by $x$ the point which is labeled $p_b$ in $K$ (and is $f(p_a)$ in $f(K)$).

It is proved in~\cite{Beguin}, Proposition~I.1, that if $C, C'$ are two Cantor sets in the interior of a closed disk $D$, there is a homeomorphism of $D$ which fixes the boundary and maps $C$ to $C'$.  We cannot quite apply this result to our situation because the Cantor sets $K$ and $f(K)$ have the point $x$ on the boundary.  

Because $K$ and $f(K)$ are both Cantor sets, the union $K \cup f(K)$ is also a Cantor set, and we can enclose it in a union of disks as follows.  Let $A_n$ be a sequence of closed disks such that:
\begin{enumerate}
\item $A_n \subseteq U$.
\item The $A_n$ are pairwise disjoint.
\item As $n \to \infty$, the set $A_n$ Hausdorff converges to the single point $x$.
\item For all $n$, $A_n$ contains some point in $K$ and some point in $f(K)$ and $\partial A_n$ does not intersect $K \cup f(K)$.
\item $(K\cup f(K)) - \{x\} \subseteq \bigcup_n A_n$.
\end{enumerate}
That is, $A_0$ contains the points in $K \cup f(K)$ which are far from $x$, then $A_1$ contains some points that are closer, and so on, limiting to the single point $x$.  Note that both $K$ and $f(K)$ have $x$ as an accumulation point, so there is no trouble containing points from both sets and also Hausdorff limiting to $x$.

Because $\partial A_n$ does not intersect $K \cup f(K)$, we must have $A_n$ containing two sub-Cantor-sets from $K$ and $f(K)$, and we can apply \cite{Beguin}~Proposition~I.1 to obtain a homeomorphism $g_n :A_n \to A_n$ taking $f(K) \cap A_n$ to $K \cap A_n$.  Since the diameter of $A_n$ is going to zero, the infinite composition  $g = \ldots \circ g_1 \circ g_0$ is a well-defined homeomorphism of $U$ which takes $f(K)$ to $K$.  Then $g \circ f$ is the desired homeomorphism in the lemma. 
\end{proof}

\subsection{Filling rays}

We will be interested in understanding how long rays interact with short rays and loops.  The concept of \emph{filling} rays is important.  Recall rays and loops are defined up to isotopy; hence a long ray intersects a ray or loop if it intersects every representative of their equivalence class.  Equivalently, the geodesic representatives intersect.

\begin{definition}
A long ray $l$ is said to be:
\begin{itemize}
\item \emph{loop-filling} if it intersects every loop.
\item \emph{ray-filling} if it intersects every short ray.
\item \emph{k-filling} if 
       \begin{enumerate}
       \item There exists a short ray $l_0$ and long rays $l_1,\ldots, l_k=l$ such that $l_i$ is disjoint from  $l_{i+1}$ for all $i\ge 0$.
       \item $k$ is minimal for this property.
      \end{enumerate}
\item \emph{high-filling} if it is ray-filling and not $k$-filling for any $k \in \N$.
\end{itemize}
\end{definition}

That is, a long ray is $k$-filling if it is at distance exactly $k$ from the set of short rays in the graph $\RGC$.  A long ray is high-filling if it is not in the connected component of $\RGC$ containing the short rays.

\subsubsection{Examples of high-filling rays}
In this section, we give some examples of high-filling rays.  We will not need the fact that these rays are actually high-filling; they are provided as examples to help understand the ideas.  Therefore, we just give a brief sketch of \emph{why} they are high-filling.

Figure~\ref{figu:high-filling_alpha_k} gives an example of a high-filling ray. It is the limit ray of the sequence $(\alpha_k)$ defined in \cite{Juliette}. This sequence is such that for every $k \geq 1$, every ray disjoint from a ray which $\textnormal{long}(\alpha_{k+1})$-begins like $\alpha_{k+1}$ has to $\textnormal{long}(\alpha_k)$-begin like $\alpha_k$, where for every ray $x$, $\textnormal{long}(x)$ denotes the number of segments of the equator that $x$ intersects.  Moreover, for every $k$, $\textnormal{long}(\alpha_{k+1})$ is greater than $\textnormal{long}(\alpha_{k})$ (these properties are established in~\cite{Juliette}). As points of the boundary of the universal cover, the endpoints of the $\alpha_k$'s are cover-converging to a point, denoted by $p$. According to Lemma~\ref{lemma:simple_is_compact}, the geodesic between $\tilde \infty$ and $p$ is the lift of a simple ray, denoted by $\alpha_\infty$. We claim that this ray $\alpha_\infty$ is high-filling. Indeed, according to the properties of $(\alpha_k)$, there is no ray disjoint from $\alpha_\infty$ (more precisely, every ray disjoint from $\alpha_\infty$ has to $\infty$-begin like $\alpha_\infty$). Thus, there are no edges in $\RGC$ incident on $\alpha_\infty$; it is a clique unto itself. As a consequence, we see that there exist infinitely many cliques in $\RGC$: any image of $\alpha_\infty$ under an element $\phi$ of $\Mod(\R^2 -\Cantor)$ which does not preserve $\alpha_\infty$ gives a clique with the one element $\phi(\alpha_\infty)$.

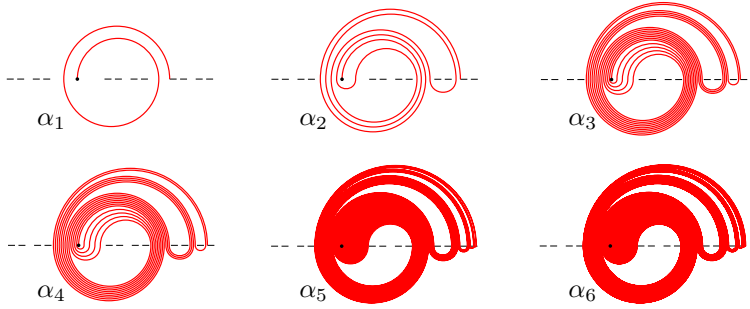
\begin{figure}[htb]
\labellist
\pinlabel $\alpha_1$ at 13 53
\pinlabel $\alpha_2$ at 89 53
\pinlabel $\alpha_3$ at 167 53
\pinlabel $\alpha_4$ at 13 3
\pinlabel $\alpha_5$ at 89 3
\pinlabel $\alpha_6$ at 167 3
\endlabellist
\centering
\includegraphics[scale=1.3]{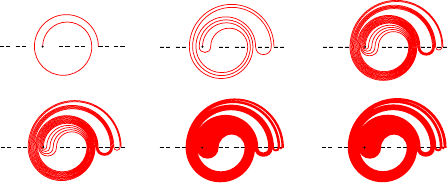}
\caption{Example of (the beginning of) a high-filling ray -- the high-filling ray itself is $\alpha_\infty$ (which looks like $\alpha_6$ if we do not zoom).}
\label{figu:high-filling_alpha_k}
\end{figure}

As the example above shows, it is nontrivial to prove that a long ray is high-filling.  We can create more examples by finding rays fixed by pseudo-Anosov mapping classes of subsurfaces.  The details of this are not important for this paper; we mention it here as a comment that we have examples of cliques of high-filling rays with an arbitrary number of elements, (see for example Figure~\ref{figu:4_clique}), but we don't know if there exist cliques of high-filling rays with infinitely many elements.

\begin{figure}[htb]
\centering
\includegraphics[scale=0.23]{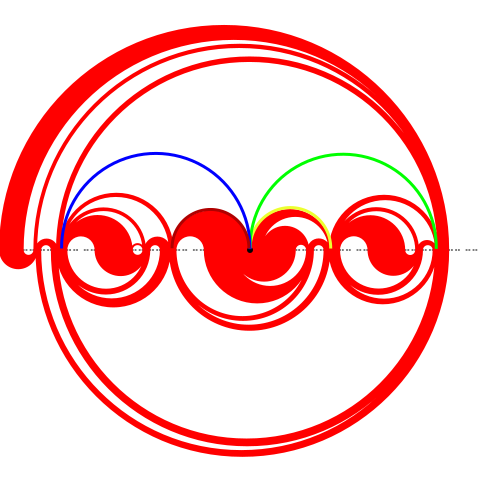}
\includegraphics[scale=0.23]{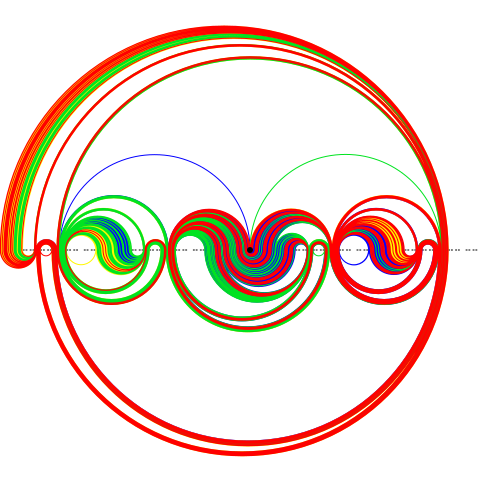}
\includegraphics[scale=0.23]{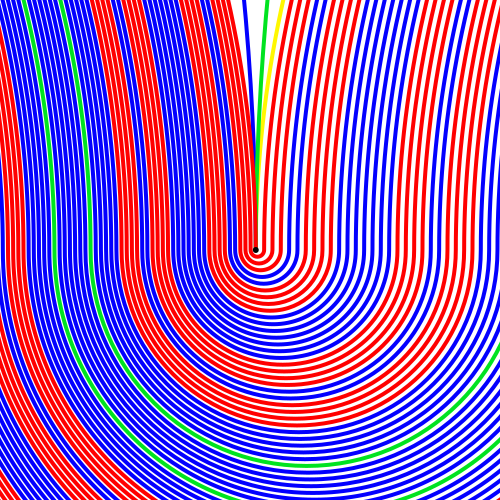}
\caption{Example of (the beginning of) four mutually disjoint high-filling long rays. One of the rays is shown in red, left, with the beginnings of the others.  We see them all (confusingly) drawn together, middle, and we show how they come together at $\infty$, right. }
\label{figu:4_clique}
\end{figure}

\subsubsection{Some properties of filling rays}

\begin{lemma}\label{lemma:loop_filling_ray_filling}
Every ray-filling ray is also loop-filling.  Hence every high-filling ray is loop-filling.
\end{lemma}
\begin{proof}
Let $r$ be a ray-filling ray.  Suppose towards a contradiction that $r$ is not loop filling, so there is some nontrivial loop $l$ which is disjoint from $r$.  Now $l$ divides $S$ into two regions $U_1, U_2$ by the Jordan curve theorem, and since $r$ and $l$ are disjoint, it must be that $r$ lies completely inside one of the two regions, say $U_1$ without loss of generality.  Since $l$ is a nontrivial loop, there are points in the Cantor set $K$ in both $U_1$ and $U_2$.  Hence, it is simple to draw a short ray totally inside $U_2$ connecting $\infty$ to some point of the Cantor set.  By construction, this ray is disjoint from $r$.  This is a contradiction.
\end{proof}

\begin{conjecture}\label{conj:loop_filling_not_ray_filling}
There are loop-filling rays which are not ray-filling.
\end{conjecture}
For an illustration of why we believe Conjecture~\ref{conj:loop_filling_not_ray_filling}, see Figure~\ref{figu:loop_filling_not_ray_filling}.

\begin{figure}[htb]
\centering
\includegraphics[scale=0.3]{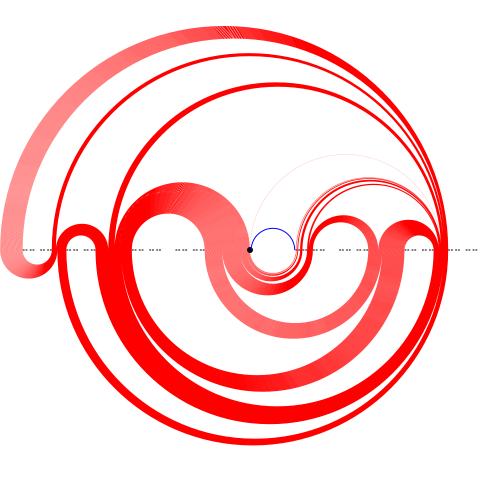}
\includegraphics[scale=0.3]{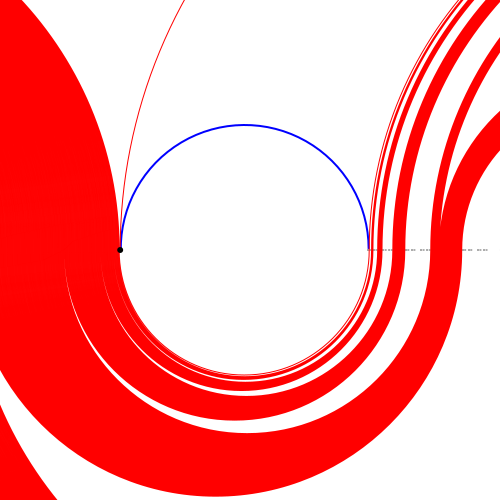}
\caption{An apparently loop-filling ray (red) which is not ray-filling.}
\label{figu:loop_filling_not_ray_filling}
\end{figure}

\begin{lemma} \label{lemma:disjoint_from_loop_filling}
Let $L$ be a loop-filling ray.
If $l$ and $l'$ are two rays (short or long) disjoint from $L$, then $l$ and $l'$ are disjoint.
\end{lemma}
\begin{proof}
Suppose that $l$ and $l'$ intersect.  We must be careful, because a priori $l'$ might accumulate on $l$ in such a way that there is no ``first'' intersection as we follow $l$ from $\infty$. However, we can still let $x$ be a point of intersection, and let $l_x$, $l'_x$ be the initial segments of $l$ and $l'$ from $\infty$ to the point of intersection $x$.  Let $y$ be the first intersection point between $l_x$ and $l'_x$ as we follow $l_x$ from $\infty$.  Let $l_{x,y}$ be the initial segment of $l_x$ from $\infty$ to $y$, and let $l'_{x,y}$ be the similar initial segment of $l'_x$.  Since $l$ and $l'$ are in minimal position, $l_{x,y}\cup l'_{x,y}$ is a nontrivial loop, and it is simple because $y$ is that first intersection of $l_x$ and $l'_x$.  This nontrivial loop is disjoint from $L$, which is a loop-filling ray.  This is a contradiction, so we conclude that $l$ and $l'$ are disjoint.
\end{proof}

\begin{corollary} \label{corollary:disjoint_from_loop_filling}
Let $l$ be a loop-filling ray which is not ray-filling, then every long ray disjoint from $l$ is not ray-filling.
\end{corollary}
\begin{proof}
Since $l$ is not ray filling, there is a short ray $x$ disjoint from $l$.  Now suppose we are given a long ray $l'$ disjoint from $l$.  As $l$ is loop-filling, according to Lemma~\ref{lemma:disjoint_from_loop_filling}, $x$ and $l'$ are disjoint. Hence $l'$ is not ray-filling.
\end{proof}

\begin{lemma}\label{lemma:no_k_filling}
There exists no $k$-filling ray for $k> 2$.
\end{lemma}
\begin{proof}
Suppose there is a $k$-filling ray $l$ for $k>2$.  Then there is a short ray $l_0$ and long rays $l_1,\ldots,l_k=l$ such that $(l_0,l_1,...,l_k)$ is a path in the graph $\RGC$.  Consider the portion of the sequence $l_0,l_1,l_2,l_3,\ldots$.  Now, $l_2$ must be ray-filling (or else we could shorten the sequence by removing $l_1$) and hence loop-filling.  Applying Lemma~\ref{lemma:disjoint_from_loop_filling}, we find that $l_3$ and $l_1$ must be disjoint.  But in this case, we can remove $l_2$ from the sequence entirely, and in fact $l$ is $(k-1)$-filling.  This is a contradiction because $l$ was assumed to be $k$-filling.
\end{proof}

Note that Lemma~\ref{lemma:no_k_filling} implies that any long ray which is in the connected component of $\RG$ inside $\RGC$ must be at distance at most $2$ from $\RG$.

\begin{question}
Do there exist $2$-filling rays?
\end{question}

Note that this question is equivalent to asking whether ray-filling rays are necessary high-filling.  We do not know the answer.  It seems difficult to find a long ray which intersects every short ray but is disjoint from a long ray which is disjoint from a short ray.

\begin{lemma}\label{lemma:cliques}
Any connected component of $\RGC$ containing a high-filling ray is a clique of high-filling rays.
\end{lemma}
\begin{proof}
Let $C$ be a connected component of $\RGC$ containing a high-filling ray.  By the definition of high-filling rays, $C$ cannot contain any short rays.  This implies it cannot contain any loops (which always have some disjoint ray).  So any two elements of $C$ are high-filling rays.  Lemma~\ref{lemma:disjoint_from_loop_filling} implies that if two long rays $a$,$b$ are disjoint from $l$, then $a$,$b$ are also disjoint from each other.  Thus $C$ is a clique. 
\end{proof}

\begin{lemma}\label{lemma:cliques_compact}
As a subset of the boundary of the conical cover, every clique of high-filling rays is compact.
\end{lemma}
\begin{proof}
The boundary of the conical cover is compact, so it suffices to show that every clique is closed under cover-convergence.  To see this, suppose we have a sequence $(l_i)$ of high-filling rays, all pairwise disjoint, which cover-converges to a ray or loop $l$.  By Lemma~\ref{lemma:simple_is_compact}, $l$ must be simple, and by Lemma~\ref{lemma:cover_converge_limit}, it must be disjoint from all the $l_i$.  But since $l$ is disjoint from all (any) of the $l_i$, it cannot be a short ray or loop because all the $l_i$ are high-filling.  Therefore, $l$ is a high-filling ray, and as it is disjoint from all (any) of the $l_i$, it is in the same clique.  So the clique is closed under cover-convergence, and we are done.
\end{proof}

\subsection{Connected components of $\RGC$}

It was shown in \cite{Juliette} that the loop graph and the ray graph are quasi-isometric. More precisely, a quasi-isometry from the ray graph to the loop graph is given by any map which sends a short ray to a disjoint loop (see \cite{Juliette}, section 3.2). This implies that the natural inclusion from the loop graph to the short-ray-and-loop graph is a quasi-isometry. Thus the ray graph, the loop graph, and the short-ray-and-loop graph are quasi-isometric. We will now show that these graphs are also quasi-isometric to the connected component of $\RG$ inside $\RGC$.

\begin{figure}[htb]
\labellist
\endlabellist
\centering
\includegraphics[scale=0.4]{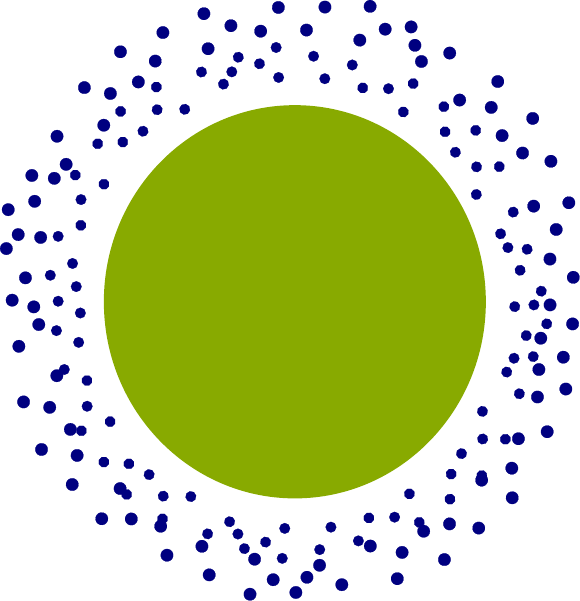}
\caption{A representation of the connected components of the completed ray graph: the main one is infinite diameter -- it contains the short-ray-and-loop graph and is quasi-isometric to it; others are cliques, of diameter $0$ or $1$. We will give a bijection between the Gromov-boundary of the main one and the set of cliques (Theorem~\ref{theorem:boundary_bijection}).}
\label{figu:completed_ray_graph}
\end{figure}

\begin{theorem}\label{theorem:completed_components}
There is a single connected component of $\RGC$ containing all the short rays and loops.  This component is quasi-isometric to the ray graph $\RG$, to the loop graph, a to the short-ray-and-loop graph.  All other connected components (which are the high-filling rays) are cliques (complete subgraphs).
\end{theorem}

We start with a technical lemma.

\begin{lemma}\label{lemma:short_rays_enough}
Let $a$ and $b$ be two loops at distance $3$ in the completed ray graph. Then they are at distance $3$ in the short-ray-and-loop graph.
\end{lemma}
\begin{proof}
Let $(a,\lambda,\mu,b)$ be a path in the completed ray graph.  We will show how to replace $\lambda$ and $\mu$ with short rays, which will prove the lemma.  This proof proceeds by chasing down various cases.

Recall the equator is divided into segments by the Cantor set.  For each of $\lambda$ and $\mu$, there are two cases depending on whether they intersect finitely many or infinitely many equatorial segments.  Suppose that $\lambda$ intersects infinitely many segments.  As $a$ and $b$ are loops, they can only intersect finitely many segments.  Hence there is some segment $I$ which $\lambda$ intersects but $a$ and $b$ do not.  We can therefore connect $\lambda$ to a point in the Cantor set with a path disjoint from $a$ and $b$, producing a short ray (this short ray is disjoint from $a$ because $\lambda$ is, but might not be disjoint from $b$).  If $\mu$ intersects $I$ as well, so much the better: we can connect it as well, producing two short rays.  See Figure~\ref{figu:lemma_short_rays_are_enough_1}. 

\begin{figure}[htb]
\centering
\includegraphics[scale=2.0]{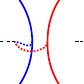}
\caption{If $\lambda$ (and possibly $\mu$) intersect an equatorial segment which neither $a$ nor $b$ do, then we can connect their first intersections to the Cantor set with disjoint paths.  Note $\lambda$ and $\mu$ are allowed to crash through this picture in any way; as long as we take the first points of intersection, the resulting short rays are disjoint}.
\label{figu:lemma_short_rays_are_enough_1}
\end{figure}

 In the latter case, we are done: we have produced a path of length $3$ between $a$ and $b$ in the short-ray-and-loop-graph.  In the former case, we may now assume that $\lambda$ is actually a short ray intersecting the equator finitely many times.  If $\mu$ intersects infinitely many segments, we repeat the argument and we are done.

Hence we have two cases: $\lambda$ and $\mu$, long rays, both intersect finitely many equatorial segments, or $\lambda$ is actually a short ray and $\mu$ intersects finitely many segments.  Potentially by switching the roles of $\lambda$ and $\mu$, then, we may assume that $\lambda$ is a long ray intersecting finitely many segments.  Since $\lambda$ intersects finitely many segments, there must be some segment $I$ it intersects infinitely many times. Since $a$ and $b$ intersect the equator finitely many times, there must be some subsegment $I'\subseteq I$ such that $\lambda$ intersects $I'$ at least twice, and $a$, $b$ intersect it not at all.   Let $p_1,p_2$ be the first and second points of intersection of $I'$ and $\lambda$ as we travel along $\lambda$ from~$\infty$.  The union of the interval along $\lambda$ between $p_1$ and $p_2$ and the interval along $I'$ between $p_1$ and $p_2$ is a simple closed curve $\gamma_\lambda$.  This curve divides the sphere into two connected components, one of which does not intersect $a$ (because $a$ is disjoint from both $I'$ and $\lambda$.  Call this component $D_\lambda$.  Note that $D_\lambda$ must contain some points in the Cantor set because $\lambda$ is geodesic.

There are several cases depending on the status of $\mu$.  If $\mu$ is disjoint from $D_\lambda$, then we can easily connect $\lambda$ to a Cantor set point in the interior with a path.  This replaces $\lambda$ with a short ray, and we can swap the roles of $\lambda$ and $\mu$ and go back to the beginning of this case to handle $\mu$.  See Figure~\ref{figu:lemma_short_rays_are_enough_2}.

\begin{figure}[htb]
\centering
\labellist
\pinlabel $\lambda$ at 11 67
\pinlabel $a$ at 11 62.5
\pinlabel $b$ at 11 58
\pinlabel $I$ at 11 54
\pinlabel $D_\lambda$ at 45 10
\pinlabel $p_1$ at 36.5 44
\pinlabel $p_2$ at 45.5 44
\endlabellist
\includegraphics[scale=1.9]{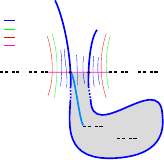}
\caption{If $\mu$ is disjoint from $D_\lambda$, we just replace $\lambda$ 
with a short ray which simply dead ends to any Cantor set point in $D_\lambda$ 
(the change is shown in lighter blue).}
\label{figu:lemma_short_rays_are_enough_2}
\end{figure}
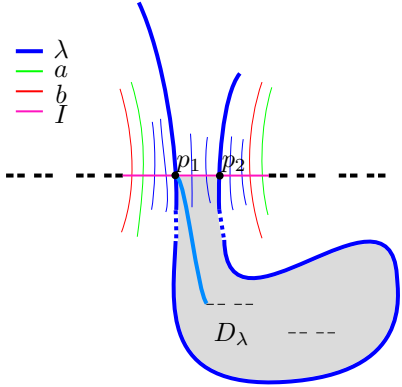

If $\mu$ does intersect $D_\lambda$, then we must consider $b$.  Note that $b$ can intersect $D_\lambda$ finitely many times.  Hence there is a subdisk $D_\lambda' \subseteq D_\lambda$ disjoint from $b$ and containing the interval $I'$.  Since $\mu$ is disjoint from $\lambda$ and $b$, it must be that $\mu \cap D_\lambda$ is contained within $D_\lambda'$.  As $\mu$ is geodesic, we must have some Cantor set points within $D_\lambda'$.  Since $a$ and $b$ are both disjoint from $D_\lambda'$, we can connect both $\lambda$ and $\mu$ to Cantor set points with paths in $D_\lambda'$.  This replaces $\lambda$ and $\mu$ with short rays $\lambda'$, $\mu'$ with the same disjointness properties, producing the desired path in the short-ray-and-loop graph.  See Figure~\ref{figu:lemma_short_rays_are_enough_3}.

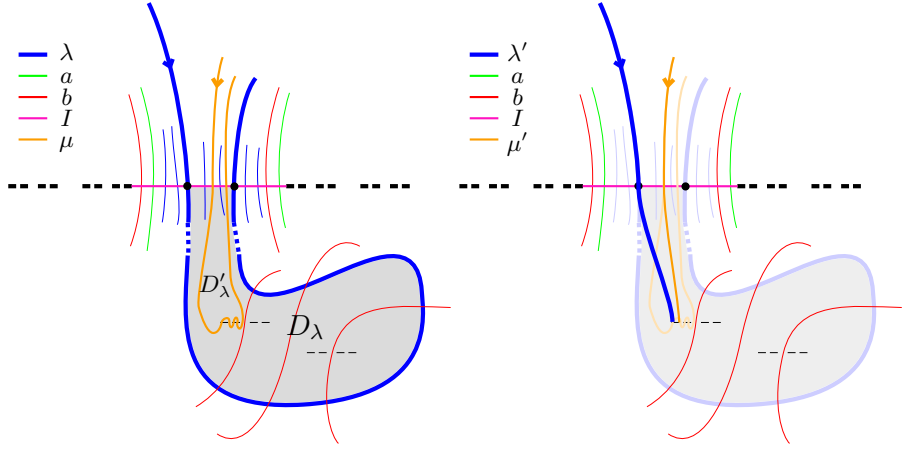
\begin{figure}[htb]
\centering
\labellist\small
\pinlabel $\lambda$ at 11 74
\pinlabel $a$ at 11 69
\pinlabel $b$ at 11 65.5
\pinlabel $I$ at 11 61.5
\pinlabel $\mu$ at 11 56.5
\pinlabel $D_\lambda'$ at 39 30
\large
\pinlabel $D_\lambda$ at 56 22
\endlabellist
\includegraphics[scale=2.0]{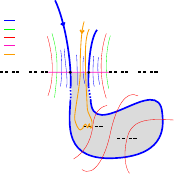}
\labellist\small
\pinlabel $\lambda'$ at 11 74
\pinlabel $a$ at 11 69
\pinlabel $b$ at 11 65.5
\pinlabel $I$ at 11 61.5
\pinlabel $\mu'$ at 11 56.5
\endlabellist
\includegraphics[scale=2.0]{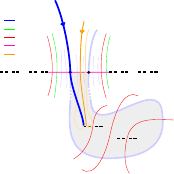}
\caption{If $\mu$ is not disjoint from $D_\lambda$, then we create $D_\lambda'$.  Here we can short-cut both $\lambda$ and $\mu$, as shown on right.  Note it only matters that the new paths are disjoint from each other, not the original $\mu$.  For example, $\mu$ could even be ray-filling in $D_\lambda'$.}
\label{figu:lemma_short_rays_are_enough_3}
\end{figure}

\end{proof}

\begin{proof}[Proof of Theorem~\ref{theorem:completed_components}]
We recall from~\cite{Juliette} the definition of a unicorn path between two loops in the loop graph (or, see Section~\ref{section:finite_unicorn_paths}).  The precise definition is not important: all we need is that this operation produces a path between any two loops.  Thus the loop graph is connected.

By~\cite{Juliette}, we already know that the ray graph $\RG$ is quasi-isometric to the loop graph and short-ray-and-loop graph.  The loop graph is connected, so these other graphs are too.  Because these graphs are quasi-isometric, we may conflate their definitions and show that the natural inclusion $\RG \to \RGC$ of the short-ray-and-loop graph into its connected component in $\RGC$ is a quasi-isometry.  Note all this means is that we must show that lengths in $\RG$ are not distorted when we add the ability to take ``detours'' using long rays in $\RGC$.  First we show it is a quasi-isometric embedding and then prove it is quasi-surjective.

Consider two loops or short rays $a$ and $b$ at distance $n \in \N$ in the completed ray graph. There exists two loops $a'$ disjoint from $a$ and $b'$ disjoint from $b$. Consider a geodesic path $(\lambda_i)$ between $a'$ and $b'$ in the completed ray graph. We claim that every element of this path is not loop-filling: by contradiction, if some $\lambda_j$ is loop-filling, then because $\lambda_j$ is different from $a'$ and $b'$, $\lambda_{j-1}$ and $\lambda_{j+1}$ are well-defined vertices in the path, and according to Lemma~\ref{lemma:disjoint_from_loop_filling}, they are disjoint. This gives a contradiction because $(\lambda_i)$ is geodesic.

For every $i$ such that $\lambda_i$ is well-defined, we choose a loop $a_i$ disjoint from $\lambda_i$. According to Lemma~\ref{lemma:short_rays_enough}, $a'_i$ and $a'_{i+1}$ at at distance at most $3$, hence the distance between $a'$ and $b'$ in the short-ray-and-loop graph is at most $3$ times the distance between $a'$ and $b'$ in the completed ray graph. Hence we have a path of length $3n+2$ between $a$ and $b$ in the short-ray-and-loop graph.  Thus we conclude that $\frac{1}{3}d_\RG(a,b) -2 \le d_\RGC(a,b) \le d_\RG(a,b)$ so the inclusion is a quasi-isometry.  Lemma~\ref{lemma:no_k_filling} immediately implies that this inclusion is quasi-surjective.

Finally, Lemma~\ref{lemma:cliques} says that all other connected components  of $\RGC$ must be cliques because they are composed of high-filling rays.
\end{proof}

\begin{remark}
After Theorem~\ref{theorem:completed_components}, we will conflate the definitions of the ray graph, the loop graph, and the short-ray-and-loop graph, using whichever is most convenient.  We'll refer to all of these graphs, and also to the quasi-isometric connected component of $\RGC$, by $\RG$.  We'll also call this component the \emph{main component}, as distinguished from the \emph{cliques}.  We show in Section~\ref{section:bijection} that the
cliques are in bijection with the Gromov boundary of $\RG$.
\end{remark}

\section{Infinite unicorn paths}\label{section:unicorn_paths}

\subsection{Outline}

We are interested in understanding the boundary of $\RG$.  We will use the same general idea as \cite{Pho-On} of \emph{infinite unicorn paths}, although with a slightly different definition.  That idea, in turn, was motivated by unicorn paths in the curve graph, as defined in \cite{Hensel-Przytycki-Webb}.  See \cite{Juliette} for an introduction to unicorn paths in the loop graph.  In this section, it is most convenient to work with the loop graph, although we recall that there is a bijection between its boundary and the boundary of the ray graph (and the short-ray-and-loop graph and the main component of $\RGC$).

\subsection{Finite unicorn paths}
\label{section:finite_unicorn_paths}

There are two equivalent definitions of a finite unicorn path.  We choose to use a slightly different definition than that which appears in~\cite{Juliette, Hensel-Przytycki-Webb} because it has some nice properties.  However, it is therefore important for us to explain both definitions and prove their equivalence.

\begin{definition}[Unicorn path, first definition, \cite{Hensel-Przytycki-Webb,Juliette}]\label{def:unicorn_1}
Let $a,b$ be loops representing vertices in the loop graph.  Let them be oriented.  We define the \emph{unicorn path} $P_1(a,b)$ between $a$ and $b$ as follows.  If $a,b$ are disjoint, then the unicorn path is simply the pair $(a,b)$.  If $a,b$ intersect, then we order their points of intersection $x_1, x_2, \ldots, x_k$ in the reverse order they appear along $a$.  We define $a_i = (\infty x_i)_b \cup (\infty x_i)_a$; that is, the loop which is the union of the interval along $b$ from $\infty$ to $x_i$ and the interval along $a$ from $\infty$ to $x_i$ (the beginnings of both $a$ and $b$ until they reach $x_i$).  It is possible that this loop is not simple.  The unicorn path $P_1(a,b)$ is the sequence $(a,a_{i_1}, a_{i_2}, \ldots, a_{i_n}, b)$, where $i_1, i_2, \ldots$ 
are the indices of the simple loops defined above.  That is, we define all the loops as above, then select just the simple ones, and that is the unicorn path.  Note that because we order the points of intersection in the reverse order they appear along $a$, we will get a little less of $a$ at each step along the path, and thus we interpolate between $a$ and $b$
\end{definition}

\begin{definition}[Unicorn path, second definition]\label{def:unicorn_2}
Let $a,b$ be loops representing vertices in the loop graph.  Let them be oriented.  We define the \emph{unicorn path} $P_2(a,b)$ as follows.  Set $a_0 = a$, and proceed inductively.  Let $x_k$ be the first intersection between $b$ and $a_k$ as we travel along $b$ from $\infty$.  If such an intersection does not exist, then $a_{k+1} = b$ and we are done.  Otherwise, set $a_{k+1} = (\infty x_k)_{b} \cup (\infty x_k)_{a_k}$.  That is, the union of the interval along $b$ from $\infty$ to $x_k$ followed by the reverse of the interval along $a_k$ from $\infty$ to $x_k$.  Note that $a_k$ is always simple, and $a_k, a_{k+1}$ are always disjoint.  The unicorn path is the sequence $P_2(a,b) = (a=a_0, a_1, \ldots, b)$.
\end{definition}

\begin{lemma}\label{lemma:unicorn_defs_equiv}
Definitions~\ref{def:unicorn_1} and~\ref{def:unicorn_2} are equivalent.  If $P_1(a,b) = (a=x_0, x_1, \ldots, x_n, b)$ and $P_2(a,b) = (a=y_0, y_1, \ldots, y_m, b)$, then $n=m$ and $x_i=y_i$ for all $i$.
\end{lemma}
\begin{proof}
Although the definition of $P_2$ is inductive, we can still think of each element of the unicorn path (second definition) as the union of a beginning interval of $b$ and a beginning interval of $a$ (if we simply don't deform $a_k$, it is clearly the union of an interval on $a$ and an interval on $b$; $a_{k+1}$ is created by simply going farther along $b$ before changing to $a$).  Hence each $y_i$ does appear in the sequence $P_1(a,b)$.  We will prove that $x_i=y_i$ for all $i$ inductively.  Clearly $x_0=a=y_0$ by definition.  Now suppose that $x_i = y_i$.  Write
\[
y_i = (\infty p)_b \cup (\infty p)_a \qquad \textnormal{and} \qquad y_{i+1} = (\infty p')_b \cup (\infty p')_a
\]
That is, $y_i$ is constructed by following $b$ until the intersection point $p$ and then $a$ backward, and $y_{i+1}$ is constructed by following $b$ until the point $p'$ (which is farther along $b$, note) and then $a$ backward.  Note that in the interval $(p,p')$ along $a$, we might find other points of intersection.  If any of these points of intersection can produce a simple loop, then these loops will appear between $x_i=y_i$ and $y_{i+1}$ in the first unicorn path $P_1(a,b)$.  Thus if we can show that all these loops are, in fact, non-simple, then we will know that $y_{i+1}=x_{i+1}$.

\begin{figure}[htb]
\labellist
\pinlabel $a$ at 10 29
\pinlabel $b$ at 40 20
\pinlabel $p$ at 32 33
\pinlabel $p'$ at 92 37
\pinlabel $q$ at 71 37
\pinlabel $y_i$ at 34 60
\pinlabel $y_{i+1}$ at 13 60
\endlabellist
\includegraphics[scale=1.7]{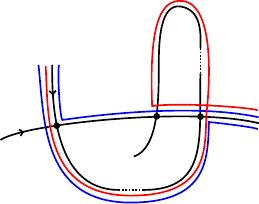}
\caption{The proof of Lemma~\ref{lemma:unicorn_defs_equiv}.  Any intersection of $b$ with the interval $(p,p')$ along $a$ must come after $p'$.  Hence any loop produced by Definition~\ref{def:unicorn_1} from any of these intersections will not be simple, and since $y_i=x_i$, we must have $y_{i+1}=x_{i+1}$.}
\label{figu:unicorn_defs_are_the_same}
\end{figure}
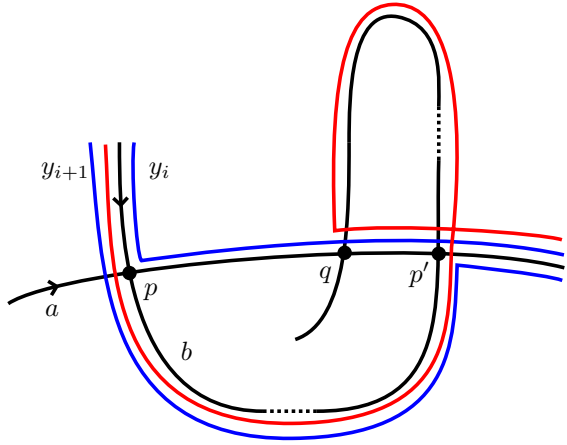

Let $q$ be a point between $p$ and $p'$ along $a$.  Were $q$ to appear before $p$ along $b$, then $y_i$ would not be simple (it would have a self-intersection at $q$); this is a contradiction to our assumption that it \emph{is} simple, so we know that $q$ must appear \emph{after} $p'$ along $b$.  But now if we take the union 
\[
(\infty q)_b \cup (\infty q)_a
\]
then this loop cannot be simple, as it must have an essential self-intersection at $p'$.  See Figure~\ref{figu:unicorn_defs_are_the_same}.

Therefore no intersection point in the interval $(p,p')$ along $b$ can contribute a loop to the sequence $P_1(a,b)$, since none of them are simple.  We conclude that $x_{i+1}=y_{i+1}$ and the lemma is proved by induction.
\end{proof}

Following Lemma~\ref{lemma:unicorn_defs_equiv}, we may now freely use our preferred Definition~\ref{def:unicorn_2} but cite facts about Definition~\ref{def:unicorn_1}.  Hereafter, we define \emph{the} unicorn path $P(a,b)$ to be $P_2(a,b)$ and we will assume that construction.  But we may use the following lemmas, which were were adapted in~\cite{Juliette} from lemmas in~\cite{Hensel-Przytycki-Webb}.

\begin{lemma}[Lemma $3.3$ of \cite{Juliette}]\label{lemma:subpaths_of_unicorn_paths}
For every $0\leq i <j\leq n$, either $P(a_i,a_j)$ is a subpath of $P(a,b)=\{a=a_0,a_1,...,a_n=b\}$, or $j=i+2$ and $a_i$ and $a_j$ represent adjacent vertices of the loop graph.
\end{lemma}

\begin{lemma}[Proposition $3.5$ and Corollary $3.6$ of \cite{Juliette}]\label{lemma:unicorn_paths_close_to_geo}
If $g$ is a geodesic of the loop graph between $a$ and $b$, then $P(a,b)$ is included the $6$-neighborhood of $g$, and the Hausdorff distance between $g$ and $P(a,b)$ is at most $13$.
\end{lemma}

An immediate consequence of the second lemma is that unicorn paths are close to quasi-geodesics.

\begin{lemma}\label{lemma:unicorn_paths_close_to_quasi}
For any $\kappa$, $\epsilon$, there exists $C \in \N$ such that if $(x_n)$ is a $(\kappa, \epsilon)$-quasi-geodesic in the loop graph, then for every $n,k$ with $n>k$, the unicorn path $P(x_0, x_k)$ is contained in the $C$-neighborhood of the unicorn path $P(x_0, x_n)$.
\end{lemma}
\begin{proof}
By Lemma~\ref{lemma:unicorn_paths_close_to_geo}, unicorn paths are uniformly close to geodesics, and for any $\kappa$, $\epsilon$ there is a constant $C'$ such that $(\kappa,\epsilon)$-quasi-geodesics are $C'$-close to geodesics.  This constant is called the $(\kappa, \epsilon)$-Morse constant of the loop graph; see Section~\ref{section:Gromov_boundaries}.  Putting these together yields the lemma.
\end{proof}

One way to think about the unicorn path $P(a,b)$ is that it slowly begins more and more like $b$.  We can make this precise with the following lemma.  Recall Definition~\ref{def:k_begin}, which defines the notion of $k$-beginning-like.

\begin{lemma}\label{lemma:uinicorn_path_begin_like_b}
Let $a$ be a loop which intersects the equator $p$ times.  For any loop or ray $b$, enumerate the unicorn path $P(a,b) = (x_n)$.  Then for $i>k\lceil p/2\rceil$, the loop $x_i$ must $k$-begin like $b$.
\end{lemma}
\begin{proof}
Divide $a=x_0$ into segments corresponding to which hemisphere of the sphere they are in.  Since $a$ intersects the equator $p$ times, there are at most $\lceil p/2\rceil$ in either hemisphere.  In particular, $b$ can intersect $a$ at most $\lceil p/2\rceil$ times before crossing the equator.  By the definition of unicorn paths, this implies that after at most $k\lceil p/2\rceil$ terms in the unicorn path, elements of the unicorn path will align with $b$ for at least $k$ equator crossings; that is, they will $k$-begin like $b$.
\end{proof}

Note that it is certainly possible that the unicorn path does not \emph{have} $\lceil p/2\rceil$ terms in it, in which case this lemma is vacuous.  The lemma just asserts that if the unicorn path is long enough compared to the number of intersections of $a$ with the equator, then elements of the unicorn path must start beginning like $b$.

\subsection{Infinite unicorn paths}

We now adapt the idea of unicorn paths to build unicorn paths between a loop and a long ray.  This idea is similar to the infinite unicorn paths defined by Witsarut Pho-On in~\cite{Pho-On}, Section~3.1.

\begin{definition}
Let $a$ be an oriented loop and $l$ a long ray.  Recall we have chosen geodesic class representatives for $a$ and $l$, so in particular they are in minimal position.  We think of $l$ as oriented away from $\infty$, and we define the unicorn path $P(a,l)$ as follows (see figure \ref{figu:infinite_unicorn_path}):
\begin{enumerate}
\item Set $a_0 = a$.
\item Let $x_{k+1}$ be the first intersection between $l$ and $a_{k}$ as we follow $l$ from $\infty$.  We define $a_{k+1} = (\infty x_{k+1})_l \cup (\infty x_{k+1})_{a_k}$.
\end{enumerate}
Note as with finite unicorn paths, $a_k$ and $a_{k+1}$ are clearly disjoint.
The unicorn path $P(a,l)$ is the (possibly infinite) sequence $P(a,l) = (a, a_1, \ldots )$.
\end{definition}

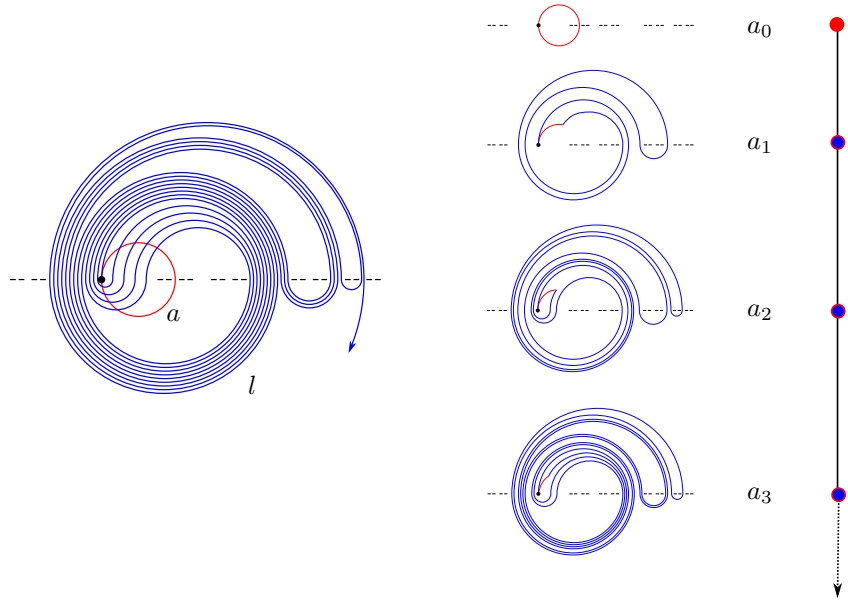
\begin{figure}[htb]
\labellist
\pinlabel $a$ at 56 97
\pinlabel $l$ at 83 73
\pinlabel $a_0$ at 257 195
\pinlabel $a_1$ at 257 155
\pinlabel $a_2$ at 257 98
\pinlabel $a_3$ at 257 35
\endlabellist
\centering
\includegraphics[scale=1.1]{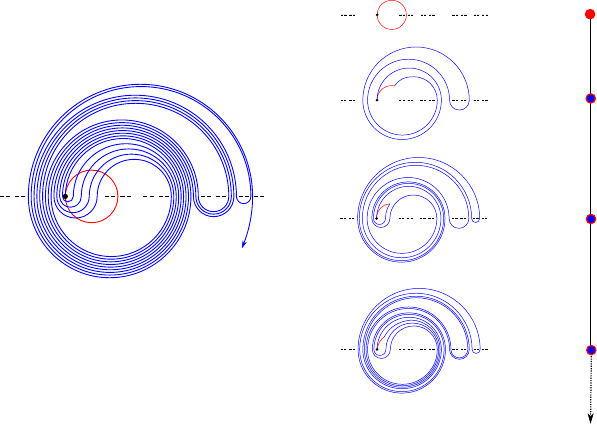}
\caption{Example of the beginning of an infinite unicorn path $P(a,l)$}
\label{figu:infinite_unicorn_path}
\end{figure}

Since we always have $a_{k+1}$ disjoint from $a_k$, $P(a,l)$ is a path in the loop graph, which we can think of as a path in $\RG$ if necessary.  In this definition, we refer to the ``first'' intersection of $a_k$ and $l$.  To address possible confusion, we note that the set of intersection points $a_k \cap l$ is discrete as a subset of $l$, although it might not be discrete as a subset of $a_k$.  That is, $l$ is a long ray, so it is free to accumulate in complicated ways in the loop $a_k$, but $a$ is a loop, so were it to accumulate anywhere in $l$, there would necessarily be points in the Cantor set arbitrary close to $a_k$ and thus \emph{on} $a_k$, which is a contradiction.

We also remark that we can always remember that $a_k$ is made up of one segment on $l$ and one segment on $a$.  As a consequence, if $b$ is disjoint from $a$, then $b$ is always disjoint from the segment $(\infty x_k)_{a_k}$, and if $b$ is also disjoint from $l$ then it is disjoint from the entire unicorn path $P(a,l)$.

\begin{remark}
It is important to note that an infinite unicorn path is so named because it is a unicorn path between a loop and an infinite object (a long ray).  It is possible that an infinite unicorn path is actually a finite sequence of loops if the long ray intersects the loop finitely many times.  For our purposes, these cases will be degenerate, so there should generally be no confusion.
\end{remark}

\begin{lemma}\label{lemma:unicorn_paths_different_loops}
Let $l$ be a long ray, and let $a$ and $b$ be two oriented loops at distance~$n$ in the loop graph. Then the unicorn paths $P(a,l)$ and $P(b,l)$ have Hausdorff distance at most $n$.
\end{lemma}
\begin{proof}
Assume first that $a$ and $b$ are disjoint and let us prove that every element of $P(a,l)$ is disjoint from some element of $P(b,l)$. The general result for $a$ and $b$ at distance $n \in \N$ follows by considering a path $(a=c_0,c_1,...,c_n=b)$ and apply this first result to each pair of disjoint loops $(c_j,c_{j+1})$.

Consider an element $a_{k}:=(\infty x_k)_l \cup (\infty x_k)_{a}$. Denote by $l_k$ the subsegment $(\infty x_k)_l\subset l$. If $l_k$ is disjoint from $b$, then $a_{k}$ is disjoint from $b$ and we are done. If $l_k$ is not disjoint from $b$, it intersects it finitely many times. Consider the first intersection point $y_k$ between $b$ and $l_k$ when we follow $b$ from $\infty$. The segment $(\infty y_k)_b$ is disjoint from $l_k$ (by definition) and from $a$ (because it is included in $b$). Hence the loop $b':=(\infty y_k)_l \cup (\infty y_k)_b$ is disjoint from $a_k$. Moreover, because $l_k$ intersects $b$ only finitely many times, $b'$ is an element of $P(b,l)$ (by definition of unicorn paths).
\end{proof}

\begin{lemma}\label{lemma:infinite_restrict_to_finite} Infinite unicorn paths restrict to finite unicorn paths. More precisely, if $a_j \in P(a,l)$ with $j\geq3$, then $P(a,a_j) \subset P(a,l)$.
\end{lemma}
\begin{proof}
This follows from the construction of $P(a,l)$ because $a_j$ contains the initial segment of $l$ which generates the first $j$ intersections with $a$.  The restriction $j \geq 3$ follows from the special case of $j=i+2$ in Lemma~\ref{lemma:subpaths_of_unicorn_paths}.
\end{proof}

\subsection{Infinite unicorn paths and cover-convergence}

\begin{lemma}\label{lemma:unicorns_k_begins} Let $(x_n)$ be a sequence of oriented loops which has a subsequence which cover-converges to a loop-filling ray $l$. Then for every $k \in \N$, there exists $n_k$ such that the unicorn paths $P(x_0,l)$ and $P(x_0,x_{n_k})$ have the same $k$-first terms.
\end{lemma}

\begin{proof}
Because $l$ is loop-filling, $P(x_0,l)$ is infinite (by definition of unicorn path). The result follows then from Lemma \ref{lemma:cover_converge_iff_k_begin} and the definition of unicorn paths.
\end{proof}

\subsection{Infinite unicorn paths and filling rays}

In this section, we relate whether or not a ray is filling to the behavior of the path $P(a,l)$.  It will turn out that the long rays whose unicorn paths give points on the boundary of $\RG$ are exactly the high-filling rays.  The definition of high-filling is in terms of rays, while the definition of unicorn paths is in terms of loops.  Therefore, we must do some translating.

Note that all distances in this section are computed in the loop graph.  This is for simplicity.  However, by Theorem~\ref{theorem:completed_components}, all statements apply equally well to the loop, ray, or short-ray-and-loop graphs, possibly with a change of constants.

\subsubsection{Bounded infinite unicorn paths}

\begin{lemma}\label{lemma:unicorn_path_non_loop_filling}
Let $l$ be a long ray which is not loop-filling. For every oriented loop~$a$, $P(a,l)$ is bounded.
\end{lemma}
\begin{proof}
Consider a loop $a$ disjoint from $l$. The only element of $P(a,l)$ is $a$, hence $P(a,l)$ is bounded. According to Lemma \ref{lemma:unicorn_paths_different_loops}, for every oriented loop $b$, $P(b,l)$ is included in the $d(a,b)$-neighborhood of $P(a,l)$, hence $P(b,l)$ is bounded.
\end{proof}

\begin{lemma}\label{lemma:unicorn_path_non_ray_filling}
Let $l$ be a long ray which is not ray-filling. Then there exists a loop $a$ such that $P(a,l)$ is included in the $2$-neighborhood of $a$.  Moreover, for every oriented loop $b$, $P(b,l)$ is bounded.
\end{lemma}

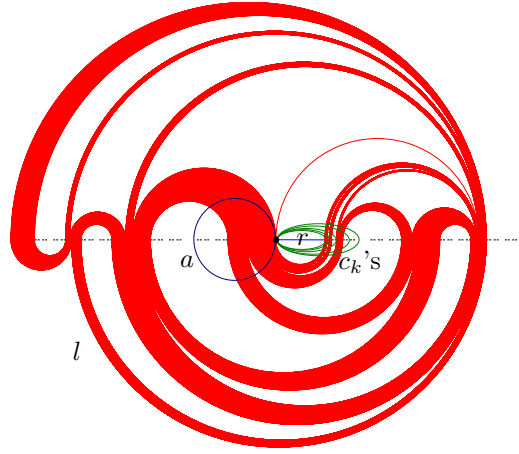
\begin{figure}[htb]
\labellist
\pinlabel $r$ at 220 157
\pinlabel $l$ at 50 73
\pinlabel $c_k$'s at 263 140
\pinlabel $a$ at 133 140
\endlabellist
\centering
\includegraphics[scale=0.5]{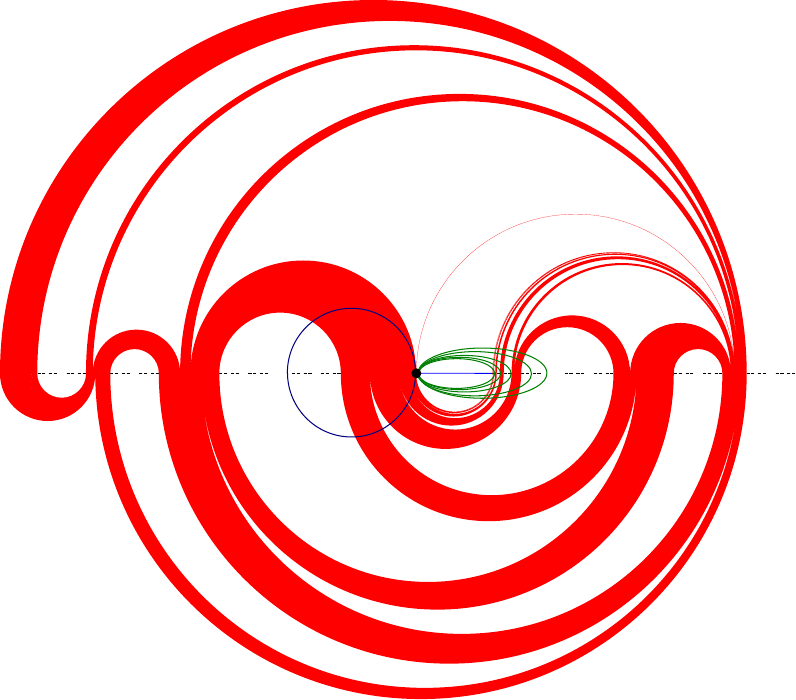}
\caption{Picture for the proof of Lemma \ref{lemma:unicorn_path_non_ray_filling}. Example of a nested sequence of loops.}
\label{figu:unicorn_nested}
\end{figure}

\begin{proof}
Let $r$ be a short ray disjoint from $l$. Consider a nested sequence of mutually disjoint loops $(c_k)$ ``around'' $r$, smaller and smaller (as in Figure \ref{figu:unicorn_nested}).

Let $a$ be a loop disjoint from all of the $c_k$ (this can be obtained by closely looping around a ray disjoint from $r$). By the construction of $(c_k)$, for every $x \in P(a,l)$, there exists $k$ such that $c_k$ is disjoint from $x$. Note that $l$ might be loop-filling, and hence the $c_k$ might not be disjoint from $l$.  However, each $x \in P(a,l)$ contains only a subsegment of $l$, followed by a subsegment of $a$, both of which are disjoint from $c_k$ for sufficiently large $k$.  Hence for all $x \in P(a,l)$ there is a $k$ so that the sequence $a-c_k-x$ is a path in the loop graph and therefore $P(a,l)$ is always within distance $2$ of $a$.

According to Lemma \ref{lemma:unicorn_paths_different_loops}, for every oriented loop $b$, $P(b,l)$ is included in the $d(a,b)$-neighborhood of $P(a,l)$, hence $P(b,l)$ is bounded.
\end{proof}

\begin{lemma}\label{lemma:unicorn_path_1_ray_filling}
Let $l$ be a $2$-filling ray. For every oriented loop $a$, $P(a,l)$ is bounded. Moreover, there exists $a$ such that $P(a,l)$ is included in the $2$-neighborhood of $a$.
\end{lemma}
\begin{proof}
By definition of $2$-filling ray, there exists a ray $l'$ disjoint from $l$ and which is not ray-filling. By Corollary~\ref{corollary:disjoint_from_loop_filling}, $l'$ cannot be loop-filling (this would imply that $l$ is not loop filling, a contradiction).  Let $a$ be a loop disjoint from $l'$. Consider any $x_k=(\infty p_k)_l \cup (\infty p_k)_a$ in $P(a,l)$. We will find a loop $y_k$ disjoint from $a$ and $x_k$: hence $x_k$ is at distance at most two from $a$.  See figure \ref{figu:unicorn_path_1_ray_filling}.

\begin{figure}[htb]
\labellist
\pinlabel $l$ at 50 145
\pinlabel $a$ at 75 130
\pinlabel $l'$ at 52 90
\pinlabel $p_k$ at 120 125

\pinlabel $C$ at 335 100
\pinlabel $l_k$ at 270 145
\pinlabel $a$ at 300 130
\pinlabel $y_k$ at 327 75
\endlabellist
\centering
\includegraphics[scale=0.7]{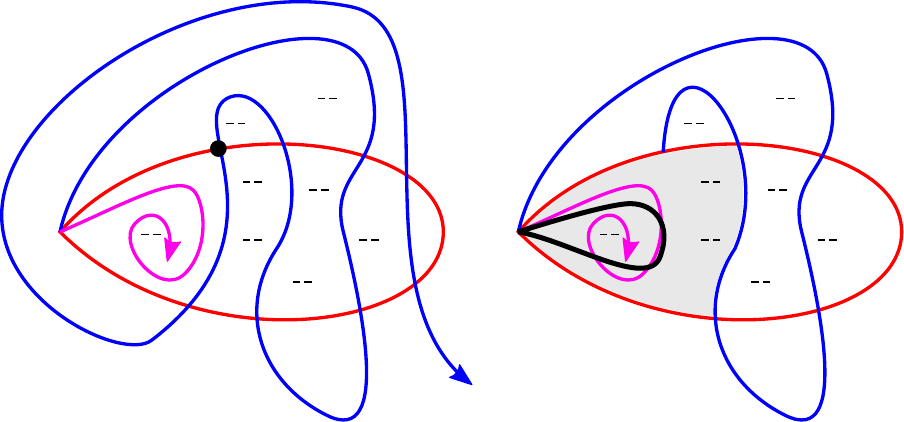}
\caption{The picture for the proof of Lemma~\ref{lemma:unicorn_path_1_ray_filling}. Here $y_k$ is drawn in black on the right; note it does not matter if it intersects $l'$ as long as it is disjoint from $a$ and $l_k$. }
\label{figu:unicorn_path_1_ray_filling}
\end{figure}
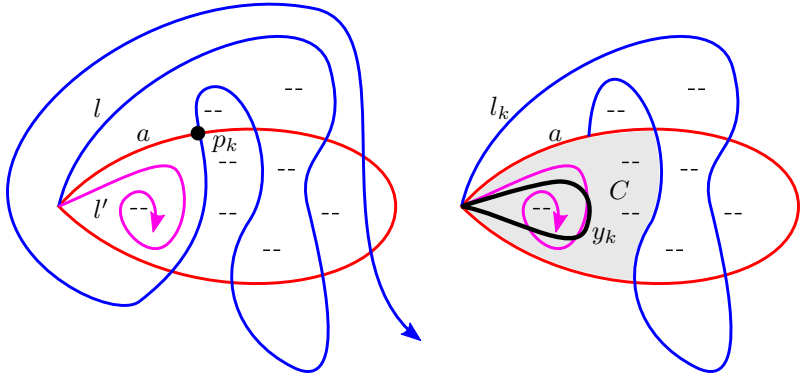

Denote by $l_k$ the sub-segment $(\infty p_k)_l$ of $l$ between $\infty$ and $p_k$. Consider $l_k \cup a$: this is a compact set of the sphere, whose complement is a finite union of open sets. Each component of the complement of $l_k \cup a$ is arcwise connected and homeomorphic to a disk or to a disk minus a Cantor set. As $l'$ is disjoint from both $l$ and $a$, it is included in a connected component of the complement of $l_k \cup a$. Denote it by $C$. Note that $C$:
\begin{itemize}
\item is not simply connected, hence contains some points of the Cantor set (because $l'$ is a geodesic included in $C$);
\item contains $\infty$ in its boundary (because it contains the ray $l'$).
\end{itemize}
It follows that we can draw a loop $y_k$ included in $C$: this loop is disjoint from both $a$ and $x_k$. Hence $P(a,l)$ is in the $2$-neighborhood of $a$. 

According to Lemma \ref{lemma:unicorn_paths_different_loops}, for every oriented loop $b$, $P(b,l)$ is included in the $d(a,b)$-neighborhood of $P(a,l)$, hence $P(b,l)$ is bounded.
\end{proof}

\subsubsection{Unbounded infinite unicorn paths}

\begin{lemma}\label{lemma:unbounded_iff_high_filling}
Let $a$ be an oriented loop and let $l$ be a long ray. The three following statements are equivalent:
\begin{enumerate}
\item The ray $l$ is high-filling.
\item The infinite unicorn path $P(a,l)$ is unbounded in $\RG$.
\item The infinite unicorn path $P(a,l)$ goes to infinity in $\RG$; that is,
if $(x_i)_i = P(a,l)$, then $\lim_{i\to\infty} d(a,x_i) = \infty$.
\end{enumerate}
\end{lemma}
\begin{proof}
First note that $(2)$ is an immediate consequence of $(3)$. We prove that $(2)$ implies $(1)$. Lemmas~\ref{lemma:unicorn_path_non_ray_filling} and~\ref{lemma:unicorn_path_1_ray_filling} show that any non-ray-filling and any $2$-filling long ray $l$ has $P(a,l)$ bounded for all $a$.  By Lemma~\ref{lemma:no_k_filling}, there is no $k$-filling ray for $k>2$.  Hence any long ray $l$ in the main component of $\RG$ (i.e. not high-filling) produces a bounded unicorn path $P(a,l)$ for any loop $a$.  So if $P(a,l)$ is unbounded, then $l$ must be high-filling.

Next, we must show that $P(a,l)$ goes to infinity if $l$ is high-filling (i.e. (1) implies~(3)).  Towards this end, let $l$ be high-filling and let a loop $a$ be given.  Suppose towards a contradiction that $P(a,l)$ does not go to infinity.  We may assume that $d(a,x)=N$ for infinitely many $x \in P(a,l)$, and by taking a further subsequence we can produce a subsequence $(y_i)$ of $P(a,l)$ so that $d(y_i,a)=N$ for all $i$.  Now for each $y_i$, we can produce a sequence $a=y_i^0, y_i^1, \ldots, y_i^N=y_i$ of loops in the loop graph such that $y_i^j$ and $y_i^{j+1}$ are disjoint.

As we have discussed, each ray or loop gives a point on the boundary $\partial\tilde{S}$ of the conical cover, which is compact.  Therefore, after possibly taking a subsequence of the $y_i$, we may assume that all of the sequences $(y_i^j)_i$ for $0\le j\le N$ cover-converges in $\partial\tilde{S}$.  These limits might be loops or rays, but whatever they are, we name them $a=y_\infty^0, y_\infty^1, \ldots, y_{\infty}^N$.  We apply Lemma~\ref{lemma:cover_converge_limit}, which says that the limits of sequences of pairs of disjoint simple rays or loops are disjoint simple rays or loops (again, we don't actually know if they are rays or loops), we observe that this sequence is a path in the completed ray graph $\RGC$ starting at $a$.  Furthermore, by the construction of an infinite unicorn path, the limit object $y_\infty^N = \lim_{i\to\infty}y_i$ must be the original ray $l$.  Hence we have produced a path in the completed ray graph $\RGC$ from $a$ to $l$.  Since $l$ is high-filling, this is a contradiction.
\end{proof}

\begin{lemma}\label{lemma:unicorn_paths_disjoint_rays}
If $l$ and $l'$ are two disjoint long rays, and $l'$ is loop-filling, then for every loop $a$, $P(a,l)$ is included in the $1$-neighborhood of $P(a,l')$.
\end{lemma}
\begin{proof}
Let $x$ be any element of $P(a,l)$. We will find $y \in P(a,l')$ such that $x$ and $y$ are disjoint. By definition, the loop $x$ is the union of a sub-segment $\bar{l}$ of $l$ and a sub-segment $\bar{a}$ of $a$. We first prove that $l'$ intersects $\bar{a}$. If not, as $l'$ is disjoint from $l$, $l'$ is in particular disjoint from $\bar{l}$, hence $l'$ is disjoint from $x$. This is a contradiction because $l'$ is loop-filling. Hence $l'$ intersects $\bar{a}$. Denote by $p$ the first intersection point between $l'$ and $\bar{a}$ when we follow $l'$ from $\infty$. The union $y=(\infty p)_{l'} \cup (p \infty)_a$ is an element of $P(a,l')$. Moreover, by construction, it is disjoint from $x$.
\end{proof}

\begin{lemma} \label{lemma:unicorn_paths_close_implies_disjoint}
Let $l,l'$ be two distinct high-filling rays, let $a,a'$ be two loops and let there be some $N\in \N$ such that $P(a,l)$ is included in the $N$-neighborhood of $P(a',l')$. Then $l$ and $l'$ are disjoint.
\end{lemma}
\begin{proof}
This proof is similar to the proof of Lemma~\ref{lemma:unbounded_iff_high_filling}.
We denote $P(a,l)$ by $(x_n)$. According to the hypothesis, for every $n$ there exists $y_n \in P(a',l')$ such that $d(x_n,y_n)\leq N$. For every $n$, we choose a path $(x_n=x^0_n,x^1_n,...,x^N_n=y_n)$ of length $N$ between $x_n$ and $y_n$. Because $(x_n)=P(a,l)$, $(x_n)$ cover-converges to $l$, and $(y_n)$ cover-converges to $l'$.  By taking subsequences, we can assume that each sequence $x^i_n$ cover-converges to some loop or ray $l_i$.  Applying the same idea as Lemma~\ref{lemma:unbounded_iff_high_filling}, we see that $l_i$ is disjoint from $l_{i+1}$, and hence $l=l_0, \ldots, l_N=l'$ is a path in the completed ray graph.  Since $l$ and $l'$ are high-filling rays, it follows from Theorem~\ref{theorem:completed_components} that $l$ and $l'$ are disjoint.
\end{proof}

\section{Gromov boundaries and unicorn paths}
\label{section:Gromov_boundaries}

We are building up to Theorem~\ref{theorem:boundary_bijection}, for which we need to understand some background about what the Gromov boundary is and how unicorn paths are related to quasi-geodesics.  In this section, we prove several key facts which will be assembled in Section~\ref{section:bijection}.

\subsection{Non-properness}

Our aim is to study the boundary of the ray graph $\RG$, or equivalently the boundary of the loop graph or main component of $\RGC$.  Recall that a metric space is proper if every closed ball is compact.  In proper Gromov hyperbolic spaces, we can readily understand the boundary using geodesic rays in the space.  Unfortunately, $\RG$, like the standard curve complex for finite-type surfaces, is not proper.  So we must be more careful.

\subsection{The Gromov boundary}

We first recall the definition of the \emph{Gromov-boundary} of a $\delta$-hyperbolic space $X$ (see \cite{Bridson-Hafliger}, or \cite{Kapovich}). Choose a basepoint $p \in X$. A sequence $(x_n)$ of elements of $X$ is said to be \emph{Gromov-converging at infinity} if:
$$\liminf_{i,j \rightarrow \infty} (x_i \cdot x_j)_p=\infty,$$
where $(x_i \cdot x_j)_p$ is the Gromov product:
$$(x_i \cdot x_j)_p:=\frac{1}{2}(d(x_i,p)+d(x_j,p)-d(x_i,x_j)).$$

Two such sequences $(x_n)$ and $(y_n)$ are said to be \emph{equivalent} if:
$$\liminf_{i,j \rightarrow \infty} (x_i \cdot y_j)_p=\infty.$$

\begin{definition}The \emph{Gromov-boundary} of a $\delta$-hyperbolic space $X$, denoted by $\partial X$, is the set of equivalent classes of sequences Gromov-converging to infinity.
\end{definition}

This definition does not depend on the choice of the basepoint $p$. If $X$ is proper and geodesic, it is known that every equivalence class of sequences Gromov-converging to infinity contains a sequence which is a geodesic of $X$. The graph $\RG$ is not proper, and we don't know if there is a geodesic in every equivalence class in its boundary. However, according to Remark $2.16$ of \cite{Kapovich}, there exists a $(1,10\delta)$ quasi-geodesic path in every element of the boundary. This quasi-geodesic will be enough for our purposes.

\begin{remark}
A quasi-isometry between Gromov hyperbolic spaces extends to a bijection of their boundaries (see for example Proposition 6.3 of \cite{Bonk-Schramm}).  
\end{remark}

\subsection{Quasi-geodesics and cover-convergence}

In this section, we show that the only way to approach the boundary of the loop graph is with a sequence cover-converging to a loop-filling ray.

\begin{lemma}\label{lemma:quasi_geodesic_to_long_ray}
Let $(x_n)$ be a quasi-geodesic sequence of loops in the loop graph.  For example, let $(x_n)$ represent a point on the Gromov boundary of the loop graph.  If $\alpha$ is a ray or loop such that some subsequence of $(x_n)$ cover-converges to $\alpha$, then $\alpha$ must be a loop-filling ray.
\end{lemma}

\begin{proof}
Recall from Lemma~\ref{lemma:simple_is_compact} that in the set of endpoints on the boundary of the conical cover of lifts of simple rays and loops, the set of endpoints of loops is isolated.  Therefore, if a subsequence of $x_n$ converges to a loop $\alpha$, that means that $\alpha$ appears infinitely often in the sequence $x_n$, which contradicts the fact that it is quasi-geodesic.  Thus $\alpha$ must be a ray, and we are left proving that $\alpha$ must be a loop-filling ray.

Suppose towards a contradiction that there exists
some loop $\beta$ disjoint from $\alpha$. Up to extending $(x_n)$ by a finite path between $\beta$ and $x_0$ if necessary, we can assume that $(x_n)$ is a quasi-geodesic which has a subsequence which cover-converges to $\alpha$ and whose first term $x_0$ is disjoint from $\alpha$. Abusing notation, we denote again by $(x_n)$ the subsequence which cover-converges to $\alpha$.

Denote by $\tilde \alpha$ and $\tilde x_n$ the lifts of $\alpha$ and $x_n$, respectively, which start at the preferred lift $\tilde{\infty}$ of $\infty$.
Consider also all the lifts of $x_0$ in the conical cover: none of them intersect $\tilde \alpha$. Because $(x_n)$ is a quasi-geodesic, there exists $N$ such that for any $n\geq N$, $x_n$ intersects $x_0$. For every such $n$, we denote by $s_n$ the first lift of $x_0$ that $x_n$ intersects.

Because $x_0$ is a loop, each preimage of each hemisphere only intersects finitely many preimages of $x_0$.  Consequently, for any $\epsilon$, there are only
finitely many lifts of $x_0$ whose endpoints on the circle are
distance greater than $\epsilon$ apart.  And because $\alpha$ is a ray, the end point of $\tilde{\alpha}$ on the boundary of the conical cover is not a lift of $\infty$.  Therefore, the sequence $s_n$ must have infinitely many distinct elements, and the endpoints must both converge to the endpoint of $\tilde{\alpha}$.
It follows that if a sequence of oriented loops $(w_n)$ is such that for every $n$, the first lift of $x_0$ intersected by $w_n$ is $s_n$, then $w_n$ cover-converges to $\alpha$.

Now we apply Lemma~\ref{lemma:unicorn_paths_close_to_quasi}: since $(x_n)$ is a subsequence of a quasi-geodesic, there is a $C$ such that for all $n,k$ with $n>k$, we have $P(x_0, x_k)$ contained in the $C$ neighborhood of $P(x_0,x_n)$. Choose $x_k$ at distance at least $C+2$ from $\alpha$. For each $n>k$ there is some loop $y_{n}$ in $P(x_0,x_{n})$ within distance $C$ of $x_k$. For every $n>k$, we can make a path of length $C$ in the loop graph between $x_k$ and $y_n$. Denote it by $x_k, z_{1,n}, z_{2,n}, \ldots, y_n$. By taking a further subsequence, we can make $z_{j,n}$ cover-converge for all $j$.  These limits give us a path $x_k, z_{1,\infty}, z_{2,\infty}, \ldots, y_\infty$ of length $C$ in the completed ray graph. Now note that because $y_n \in P(x_0,x_n)$ for every $n$ and $d(y_n,x_0)\geq 2$, by the definition of unicorn paths we have that the first lift of $x_0$ intersected by $y_n$ is $s_n$ ($x_0$ and $x_n$ must intersect, and all entries in the unicorn path $P(x_0,x_n)$ after the first follow $x_n$ until its first intersection with $x_0$). By the discussion we had in the previous paragraph, it follows that $y_\infty=\alpha$. Hence we have constructed a path of length $C$ between $x_k$ and $\alpha$, which are at distance at least $C+2$: this is a contradiction.
\end{proof}

\subsection{Infinite unicorn paths and quasi-geodesics}

In this section, we relate infinite unicorn paths to quasi-geodesics in the loop graph.  This, in turn, relates unicorn paths to the boundary via our remark above that equivalence classes on the boundary contain quasi-geodesics.  The lemma in this section is approximately saying that sequences conver-converging to a long ray are close to unicorn paths generated by that ray.

Recall that in every $\delta$-hyperbolic geodesic space $X$, quasi-geodesics stay closed to geodesics (see for example Theorem $1.7$, page $401$ of \cite{Bridson-Hafliger}). More precisely, for all $\kappa \geq 1$ and $\epsilon \geq 0$, there exists a constant $M$, called the $(\kappa,\epsilon)$-Morse constant of $X$, such that every $(\kappa,\epsilon)$-quasi-geodesic between any two points $x, y \in X$ is included in the $M$-neighborhood of any geodesic between $x$ and $y$.

\begin{lemma} \label{lemma:unicorn_paths_INFINITE_geodesics_close}
Let $\kappa,\epsilon \geq 0$. Denote by $M$ the $(\kappa,\epsilon)$-Morse constant of the loop graph. Let $(x_n)$ be a $(\kappa,\epsilon)$-quasi-geodesic path of the loop graph such that a subsequence of $(x_n)$ cover-converges to a high-filling ray $l$. Then $(x_n)$ is included in the $(13+M)$-neighborhood of the infinite unicorn path $P(x_0,l)$. 
\end{lemma}

\begin{proof}
It will be helpful during this proof to consult Figure~\ref{figu:unicorn_path_quasi}, as there are many details to keep track of.

We denote by $(y_n)$ the unicorn path $P(x_0,l)$, which is infinite and tends to infinity according to Lemma~\ref{lemma:unbounded_iff_high_filling}. Let $j \in \N$. Let us show that $x_j$ is in the $(13+M)$-neighborhood of $P(x_0,l)$. We denote by $d_j$ the distance $d(x_0,x_j)$.

Let $A \in \N$ such that $y_A$ is the first element of $P(x_0,l)$ at distance $d_j+2M+6$ from $x_0$; that is:
\begin{itemize}
\item $d(x_0,y_A)=d_j+2M+6$;
\item For every $0<k<A$, $d(x_0,y_k)<d_j+2M+6$.
\end{itemize}

\begin{figure}[htb]
\labellist
\pinlabel $x_0$ at 5 12
\pinlabel $x_N$ at 280 11
\pinlabel $x_j$ at 105 -6
\pinlabel $y_A$ at 110 62
\pinlabel $w$ at 116 41
\pinlabel $z$ at 93 23
\pinlabel $P(x_0,l)$ at 200 62
\pinlabel $P(x_0,x_N)$ at 250 41
\pinlabel $[x_0,x_N]$ at 150 23
\pinlabel $(x_n)$ at 170 -2
\endlabellist
\centering
\vspace{0.3cm}
\includegraphics[scale=1.1]{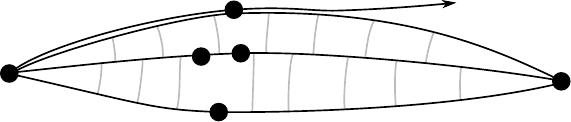}
\caption{The picture for the proof of Lemma~\ref{lemma:unicorn_paths_INFINITE_geodesics_close}.  The gray lines indicate the sequences are uniformly close.  The sequences $P(x_0,l)$ and $P(x_0,x_N)$ are actually the same until at least $y_A$.}
\label{figu:unicorn_path_quasi}
\end{figure}
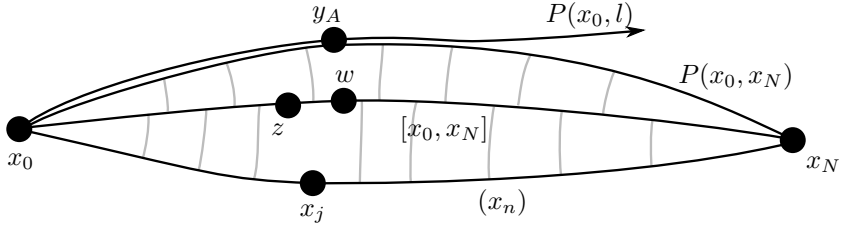

According to Lemma~\ref{lemma:unicorns_k_begins}, there exists $N$ such that $P(x_0,x_{N})$ and $P(x_0,l)$ have the same $A$ first terms. Consider a geodesic $[x_0,x_N]$. According to Lemma~\ref{lemma:unicorn_paths_close_to_geo}, the path $P(x_0,x_N)$ is included in the $6$-neighborhood of $[x_0,x_N]$. Let $w \in [x_0,x_N]$ be the closest point from $x_0$ which satisfies $d(y_A,w) \leq 6$. For every $n \leq A$, $y_n$ is $6$-close to $[x_0,w]$. Hence $[x_0,w]$ is in the $13$-neighborhood of $P(x_0,l)$. The triangle inequality gives us $d(x_0,y_A) \leq d(x_0,w)+d(y_A,w)$, hence $d(x_0,w) \geq d_j +2M$.

Because $(x_n)$ is a $(\kappa,\epsilon)$-quasi-geodesic, $x_j$ is in the $M$-neighborhood of $[x_0,x_N]$. Let $z \in [x_0,x_N]$ such that $d(x_j,z) \leq M$. The triangle inequality gives us $d(x_0,z) \leq M + d_j$. Hence $z \in [x_0,w]$. It follows that $x_j$ is at distance at most $M+13$ from $P(x_0,l)$.
\end{proof}

\subsection{Infinite unicorn paths and the Gromov boundary}

Finally, we show that unicorn paths of high-filling rays produce points on the boundary of the ray graph $\RG$.

\begin{lemma}\label{lemma:high_filling_to_infinity}
If $l$ is high-filling, then $P(a,l)$ Gromov-converges to infinity for any loop $a$.
\end{lemma}
\begin{proof}
The proof follows the steps of the proof of Proposition $3.6$ of \cite{Pho-On}. Recall that a sequence $(x_n)$ of loops Gromov-converges to infinity if
\[
\liminf_{i,j \rightarrow \infty} (x_i,x_j)_x=\infty,
\]
where $(x_i,x_j)_x$ is the Gromov product:
\[
(x_i,x_j)_x:=\frac{1}{2}(d(x_i,x)+d(x_j,x)-d(x_i,x_j)).
\]
We will prove this holds for $(a_n) := P(a,l)$. According to Lemma \ref{lemma:unbounded_iff_high_filling}, for every $R>0$, there exists $N>0$ such that $d(a,a_n)>R$ for every $n>N$. For all $m,n\geq N$, we have $(a_n,a_m)_a\geq d(a,[a_n,a_m])-2\delta$. Since $[a_n,a_m]$ and $P(a_n,a_m)$ have Hausdorff distance at most $13$ by Lemma~\ref{lemma:unicorn_paths_close_to_geo}, this implies that: 
\[
(a_n,a_m)_a \geq d(a,P(a_n,a_m))-13-2\delta > R -13-2\delta.
\]
For $|m-n|>2$, $P(a_n,a_m)$ is contained in $P(a,l)$ by Lemma~\ref{lemma:infinite_restrict_to_finite}, so we have the desired property for the sequence $P(a,l)$.  
\end{proof}

\section{The boundary of the ray graph}
\label{section:bijection}

In Sections~\ref{section:unicorn_paths} and~\ref{section:Gromov_boundaries}, we proved many technical facts about how infinite unicorn paths behave and how the filling-ness of a ray determines the behavior of its unicorn path.  In this section, we use these facts to provide a description of the boundary of the ray graph $\RG$.

\subsection{Bijection}

\begin{theorem}\label{theorem:boundary_bijection}
Given a point $p$ on the Gromov boundary $\partial \RG$, there is a nonempty set of long rays $R$ such that the point $p$ is the equivalence class of $P(a,l)$ for all $l \in R$ and any loop $a$.  The set $R$ is a clique (and an entire connected component) of high-filling rays in $\RGC$.  Conversely, given a connected component $R$ of high-filling rays in $\RGC$ (which is necessarily a clique), there is a single boundary point $p \in \partial \RG$ such that $p$ is the equivalence class of $P(a,l)$ for all $l \in R$ and any loop $a$, and $R$ is exactly the set of rays with this property.

Hence the set of high-filling cliques is in bijection with the boundary of the ray graph $\partial \RG$.
\end{theorem}

\begin{figure}[htb]
\labellist
\endlabellist
\centering
\includegraphics[scale=1]{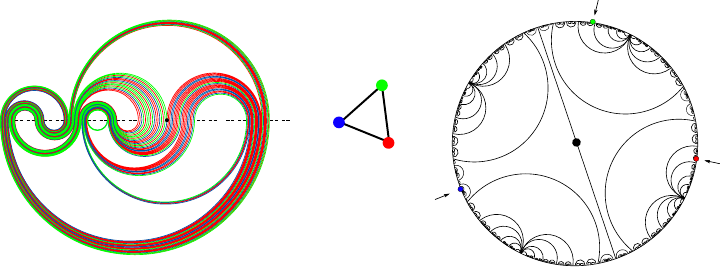}
\caption{Example of the different versions of a point of the boundary: the corresponding clique seen as disjoint high-filling rays on the surface (left), as connected component of the completed ray graph (middle), and as subset of the boundary of the conical cover (right)}
\label{figu:clique}
\end{figure}

\begin{proof}
By Lemma~\ref{lemma:high_filling_to_infinity}, given any short ray $a$ and and any high-filling ray $l$, $P(a,l)$ Gromov-converges to infinity (is in the equivalence class of a point on the Gromov boundary), and this boundary point does not depend on $a$ by Lemma~\ref{lemma:unicorn_paths_different_loops}. 

Conversely, let $p$ be a point on the boundary.  According to Remark $2.16$ of \cite{Kapovich}, there is a $(1,10\delta)$-geodesic path $(x_i)$ in the equivalence class of $p$.  Since $\partial \tilde{S}$ is compact, after passing to a subsequence we may assume that $x_i$ cover-converges to an element $l \in \partial \tilde{S}$.  By Lemma~\ref{lemma:quasi_geodesic_to_long_ray}, $l$ must be a loop-filling ray.  We claim that $l$ is a high-filling ray.  To see this, we observe that by Lemma~\ref{lemma:unicorn_paths_close_to_geo}, $P(x_0, x_i)$ stays uniformly close to the geodesic $[x_0, x_i]$, whose length goes to infinity.  By Lemma~\ref{lemma:unicorns_k_begins}, the length of $P(x_0,l)$ must then be infinite.  By Lemma~\ref{lemma:unbounded_iff_high_filling}, we must have $l$ be high-filling.  By Lemma~\ref{lemma:unicorn_paths_INFINITE_geodesics_close}, we have that $P(x_0,l)$ is a bounded distance from $(x_i)$.  Hence the equivalence class of $P(x_0,l)$ is the same as $(x_i)$, which is $p$, as desired.

We have proved that every high-filling ray has a unicorn path which produces a point on the boundary, and we have shown that for every boundary point, there is a high-filling ray whose unicorn path produces it. To complete the proof, we must show that two high-filling rays produce the same boundary point if and only if they are in the same connected component (which is a clique by Theorem~\ref{theorem:completed_components}).  The combination of Lemmas~\ref{lemma:unicorn_paths_disjoint_rays} and~\ref{lemma:unicorn_paths_close_implies_disjoint} shows that two high-filling rays are disjoint if and only if their unicorn paths stay a bounded distance apart i.e. if and only if the have the same equivalence class as boundary points.
\end{proof}

Following Theorem~\ref{theorem:boundary_bijection}, we can define a bijection between cliques in $\RGC$ and the boundary of $\RG$.  Let $\HC$ be the set of high-filling rays, and let $\E$ be the set of cliques of high-filling rays in the graph $\RGC$.  If $l$ is a high-filling ray, we let $[l]$ denote the clique in $\E$ containing $l$.  By Theorem~\ref{theorem:boundary_bijection}, 
we can define the map $F: \E \to \partial\RG$ by $F([l]) = [P(a,l)]$, and $F$ is a bijection.  That is, we take a clique in $\E$ to the equivalence class of $P(a,l)$ on the boundary, for any $l$ in the clique and any loop $a$.

\subsection{Topologies}

In this section, we define topologies on both the high-filling rays and the Gromov
boundary of the ray graph.  The topology on the high-filling rays induces a 
topology on the set of cliques in $\RGC$, and it the space $\RGC$ endowed with this
topology which we will show is homeomorphic to the Gromov boundary.
See Figure~\ref{figu:quotient} for an outline of the situation.

\begin{figure}[htb]
\labellist
\pinlabel $\HC$ at 10 120
\pinlabel $q$ at 60 90
\pinlabel $\E$ at 10 40
\pinlabel $F$ at 133 31
\pinlabel $\tilde F$ at 140 118
\pinlabel $\RG$ at 205 45
\pinlabel $\partial \RG$ at 187 95
\endlabellist
\centering
\includegraphics[scale=1]{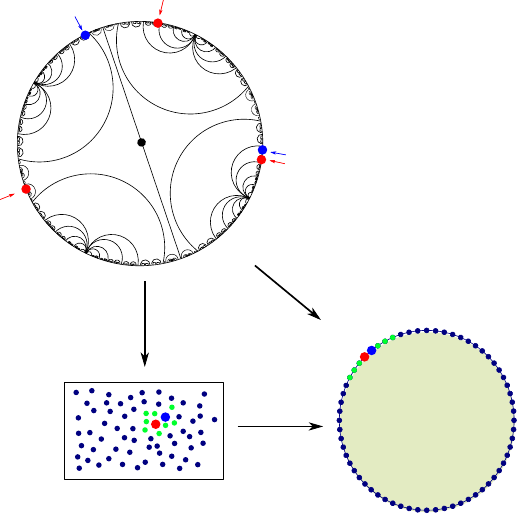}
\caption{The space $\HC$ of high-filling rays can be thought of as a subset of the boundary of the conical cover.  In the quotient $\E$, cliques are close if any of their constituent rays are close.  We will show $F: \E \to \partial \RG$ is a homeomorphism.}
\label{figu:quotient}
\end{figure}
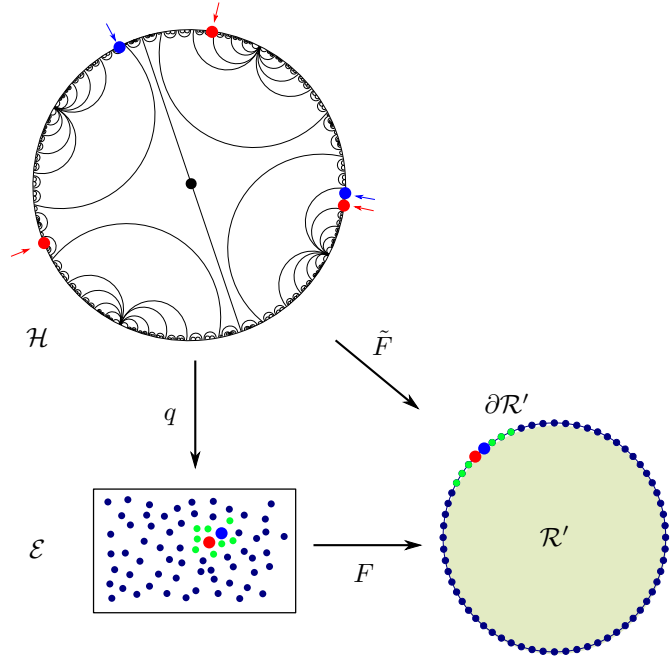

\subsubsection{Topology on the set of high-filling cliques}

There is a natural topology on the set $\HC$ of high-filling rays given by the definition of $k$-beginning-like: for every $l \in \HC$ and every $k \in \N$, we define the $k$-neighborhood $N(l,k)$ of $l$ as the set of high-filling rays which $k$-begin like $l$. The $N(l,k)$'s are a basis of open sets for $\HC$. Note that they are also closed.  Also note that this topology is the same as the topology induced by thinking about the high-filling rays as a subset of the boundary of the conical cover $\partial \tilde{S}$, i.e. as a subset of the circle.

There is a quotient map $q: \HC \to \E$ defined by $q(l) = [l]$ taking a ray $l$ to its clique.  We define the topology on $\E$ to be the quotient topology of $\HC$ under the map $q$.  Note that there is a lift of this map $\tilde{F}:\HC \to \partial\RG$ defined by $\tilde{F}(l) = [P(a,l)]$.  Clearly we have $\tilde{F} = F \circ q$.  See Figure~\ref{figu:quotient}.

\subsubsection{Topology on the Gromov boundary}

There is a standard topology on the Gromov boundary, as follows.
For every $s \in \partial \RG$ and for every $r \in \N$, we denote by $U(s,r)$ the set of boundary points $t$ such that there exist two sequences $(x_n)$ and $(y_n)$ such that $s=[(x_n)]$, $t=[(y_n)]$ and ${\liminf}_{i,j \rightarrow \infty} (x_i \cdot y_j)_p \geq r$. 

The collection $\{U(s,r)|r>0\}$ define a basis of open neighborhoods for $s$ (see~\cite{Kapovich}). 

\begin{remark} Note that if $t \in U(s,r)$, then for \emph{any} sequences $(x_n)$ and $(y_n)$ such that $s=[(x_n)]$, $t=[(y_n)]$, we have $\liminf_{i,j \rightarrow \infty} (x_i \cdot y_j)_p \geq r-2 \delta$.
\end{remark} 

The remark follows from the fact that for any $x,y,z,p$ in a $\delta$-hyperbolic space, we have $(x \cdot y)_p \geq \min \{(x \cdot z)_p,(y \cdot z)_p\}.$

It is helpful to relate this topology to the usual topology on the boundary of the conical cover.  

\begin{lemma}\label{lemma:cover_convergence_gives_boundary_convergence}
Let $(l_i)$ be a sequence of high-filling rays which cover-converges to a high-filling ray $l$.  Then $[P(a,l_i)]$ converges to $[P(a,l)]$ in the topology on the Gromov boundary $\partial \RG$.
\end{lemma}
\begin{proof}
By Lemma~\ref{lemma:cover_converge_iff_k_begin}, as $i\to\infty$, $l_i$ $k$-begins like $l$ for $k\to\infty$.  Hence $P(a,l_i)$ and $P(a,l)$ agree for the first $n$ terms, for $n \to\infty$.  Thus the Gromov product of these two sequences must to go infinity.
\end{proof}

\subsection{High-filling rays and neighborhoods}

It will be important for us to be able to relate the notion of ``close'' in terms of rays, i.e. $k$-beginning-like or cover-converging, to the notions of ``close'' in the Gromov boundary.  This is the key step in relating the two topologies defined above. First, we show that being close in the ray graph translates to being close as rays.

\begin{lemma}\label{lemma:control_of_neighborhoods_of_geodesics_general_m} Let $l$ be a high-filling ray. Let $m\in \N$. For every $k \in \N$, there exists $N \in \N$ such that every loop at distance at most $m$ from a loop which $N$-begins like~$l$ $k$-begins like one of the rays of the clique of $l$. 
\end{lemma}
\begin{proof}
Let $k$ be given, and suppose towards a contradiction that for all $N>0$, there are loops $l_N, l'_N$ such that:
\begin{enumerate}
\item The distance between $l_N$ and $l'_N$ is at most $m$.
\item $l_N$ $N$-begins like $l$
\item $l'_N$ does not $k$-begin like any ray in the clique of $l$.
\end{enumerate}
We now apply the same idea as the proof of Lemma~\ref{lemma:unicorn_paths_close_implies_disjoint}.
For each $N$, we can write the path between $l_N$ and $l'_N$ as
\[
l_N = l_N^0, l_N^1, \ldots, l_N^m = l'_N
\]
By taking a subsequence, we can assume that each sequence $(l_N^i)_N$ cover-converges to some ray or loop $l_\infty^i$.  Note that $l_\infty^0 = l$ by the construction of $l_N$.  By Lemma~\ref{lemma:cover_converge_limit}, this produces a path in the completed ray graph:
\[
l = l_\infty^0, \ldots, l_\infty^m
\]
By the construction of $l'_N$, we have that $l_\infty^m$ cannot $k$-begin like any ray in the clique of $l$.  But by Theorem \ref{theorem:completed_components}, since $l$ is high-filling, we must actually have $l_\infty^m$ \emph{in} the clique of $l$.  This is a contradiction.
\end{proof}

Next, we generalize this result to the boundary, showing that being close on the boundary means being close as rays.  The following lemma says that for any point $P$ on the boundary, if $Q$ is close to $P$ in the boundary, then for any high-filling ray representative $\lambda'$ of $Q$, there is a high-filling ray representative $\lambda$ of $P$ which begins like $\lambda'$.

\begin{proposition}\label{prop:neighborhoods_on_boundary}
Let $P \in \partial \RG$. For every $k \in \N$, there exists $N>0$ such that for every $Q \in U(P,N)$, for every $\lambda' \in \tilde{F}^{-1}(Q)$, there exists $\lambda \in \tilde{F}^{-1}(P)$ which $k$-begins like $\lambda'$.
\end{proposition}

\begin{proof} Consider $P \in \partial \RG$ and choose $k \in \N$.
Let $\lambda \in \tilde{F}^{-1}(P)$. Note that $\lambda$ is high-filling, and $[P(a,\lambda)]=P$, and the clique of high-filling rays disjoint from $\lambda$ is $\tilde{F}^{-1}(P)$.  These are all consequences of Theorem~\ref{theorem:boundary_bijection}.

According to Lemma \ref{lemma:control_of_neighborhoods_of_geodesics_general_m} with $m=2\delta + 20$, there exists $N' \in \N$ such that every loop at distance less than $2 \delta +20$ from a loop which $N'$-begins like $\lambda$ must $k$-begin like one of the rays of $\tilde{F}^{-1}(P)$. Let us prove that $N:=N'+2\delta$ satisfies the proposition.

Pick a loop $a$ which has only one intersection point with the equator. Denote by $(x_n)$ the infinite unicorn path $P(a,\lambda)$. Because $a$ has only one intersection with the equator, we have that for every $n\geq 1$, $x_n$ at least $(n-1)$-begins like $\lambda$ by Lemma~\ref{lemma:uinicorn_path_begin_like_b}. Hence $x_{N'+1}$ $N'$-begins like $\lambda$.

Choose $Q \in U(P,N'+2\delta)$ and $\lambda' \in \tilde{F}^{-1}(Q)$. Again by Theorem~\ref{theorem:boundary_bijection}, we have $\lambda'$ high filling and $[P(a,\lambda')]=Q$.  Denote by $(y_n)$ the unicorn path $P(a,\lambda')$. By definition of $U(P,N'+2\delta)$, we have $\liminf_{i,j \rightarrow \infty} (x_i \cdot y_j)_a \geq N'$. Choose $i>N'+1$ and $j \in \N$ such that $(x_i \cdot y_j)_a \geq N'$: any two geodesics $[a,x_i]$ and $[a,y_j]$ have their $N'$ first terms $(2 \delta)$-Hausdorff close.

Because geodesics and unicorn paths are uniformly close (Lemma~\ref{lemma:unicorn_paths_close_to_geo}), it follows that some point of $(y_n)$, say $y_p$, is in the $2\delta+20$ neighborhood of $x_{N'+1}$. According to Lemma \ref{lemma:control_of_neighborhoods_of_geodesics_general_m}, $y_p$ $k$-begins like one of the rays in $\tilde{F}^{-1}(P)$. Thus $\lambda'$ also $k$-begins like one of the rays in $\tilde{F}^{-1}(P)$. 
\end{proof}

\subsection{Homeomorphism}

In this section, we show that $F:\E \to \partial \RG$ is a homeomorphism, given the two topologies we have defined above.  

\begin{lemma}\label{lemma:F_continuous}
The map $\tilde F = F\circ q : \HC \rightarrow \partial \RG$ is continuous.  Consequently, $F$ is continuous.
\end{lemma}

\begin{proof}
Consider a sequence of high-filling rays $l_n$ which converges to $l$ in the topology on $\HC$.  That is, for all $k$ there is $N$ such that for all $n>N$, we have $l_n$ $k$-begins like $l$.  This is equivalent to $l_n$ cover-converging to $l$, so applying Lemma~\ref{lemma:cover_convergence_gives_boundary_convergence}, we conclude that $([P(a,l_n)])_n$ converges to $[P(a,l)]$ in the Gromov topology, i.e. $(\tilde{F}(l_n))_n$ converges to $\tilde{F}(l)$.

The quotient map $F$ is then continuous by the definition of the quotient topology. \end{proof}

\begin{lemma}\label{lemma:F_open}
The map $F : \E \to \partial \RG$ is open.
\end{lemma}

\begin{proof}
Consider an open set $\mathcal{U} \subset \E$. Let us show that $F(\mathcal{U})$ is open in $\partial X$. Let $P \in F(\mathcal{U})$ be a point of the boundary. We denote by $R=F^{-1}(P) \in \E$ the preimage of $P$. We want to prove that some neighborhood of $P$ is included in $F(\mathcal{U})$.

 Because $\mathcal{U}$ is open, by definition of the quotient topology, this means that $q^{-1}(\mathcal{U})$ is open in $\mathcal{H}$, where $q : \HC \to \E$ is the quotient from the set of high-filling rays to the set of cliques, as previously defined. It follows that for every high-filling ray $\lambda$ in $q^{-1}(R)$, there exists $k(\lambda)>0$ such that $N(\lambda,k(\lambda)) \subset q^{-1}(\mathcal{U})$, i.e.:
$$ \bigcup_{\lambda \in R} N(\lambda,k(\lambda)) \subset q^{-1}(\mathcal{U}).$$
According to Lemma \ref{lemma:cliques_compact},  the clique $R$ is compact as a subset of the boundary of the conical cover, and since each set $N(\lambda, k(\lambda))$, thought of as a subset of the (circle) boundary of the conical cover is open (the intersection of an open interval with the set of high-filling rays), it follows that there exists a finite family $(\lambda_i)_{1\leq i \leq n}$ of rays in $R$ such that:
$$R \subset \bigcup_{1\leq i \leq n} N(\lambda_i,k(\lambda_i)).$$

Note for any ray $\mu \in R$, we must have $\mu \in N(\lambda_i, k(\lambda_i))$ for some $i$.
Set $k:=\max_{1\leq i \leq n}{k(\lambda_i)}$.  
According to Proposition~\ref{prop:neighborhoods_on_boundary}, there exists $N$ such that for every $Q\in U(P,N)$, for every $\lambda' \in \tilde{F}^{-1}(Q)$, there exists $\mu \in R$ which $k$-begins like $\lambda'$.

Now we have:
\begin{enumerate}
\item $\lambda'$ must $k$-begin like $\mu$ for $\mu \in R$.
\item Since $\mu \in R$, we have $\mu \in N(\lambda_i,k(\lambda_i))$ for some $i$, so $\mu$ $k(\lambda_i)$-begins like $\lambda_i$.
\item $k > k(\lambda_i)$, so as $\lambda'$ $k$-begins like $\mu$, and $\mu$ $k(\lambda_i)$ begins like $\lambda_i$, we have $\lambda'$ $k(\lambda_i)$-begins like $\lambda_i$.
\end{enumerate}
That is, $\lambda' \in N(\lambda_i, k(\lambda_i))$.  We specifically constructed this set so that $N(\lambda_i, k(\lambda_i)) \subset q^{-1}(\mathcal{U})$.  Hence, $\lambda' \in q^{-1}(\mathcal{U})$.

We have shown that for all $Q \in U(P,N)$ and for all rays $\lambda' \in \tilde{F}^{-1}(Q)$, we have $\lambda' \in q^{-1}(\mathcal{U})$.  That is, $Q \in F(\mathcal{U})$.  We conclude that $U(P,N) \subset F(\mathcal{U})$.  Since $P \in F(\mathcal{U})$ was arbitrary and $U(P,N)$ is an open set, we conclude that $F(\mathcal{U})$ is open.  Hence $F$ is open.
\end{proof}

\begin{theorem}\label{theorem:boundary_homeo}
The map $F : \E \to \partial \RG$ is a homeomorphism.
\end{theorem}

\begin{proof}
We have proved that $F$ is a bijection (Theorem~\ref{theorem:boundary_bijection}) which is continuous (Lemma~\ref{lemma:F_continuous}) and open (Lemma~\ref{lemma:F_open}).
\end{proof}


\begin{thebibliography}{99}
\bibitem{Juliette}
	Juliette Bavard, 
	\emph{Hyperbolicit\'{e} du graphe des rayons et quasi-morphismes sur un gros groupe modulaire}, 
	Geometry \& Topology {\bf 20} (2016) 491�-535
\bibitem{Beguin}
	F. B\'{e}guin, S. Crovisier, and F. Le Roux,
	\emph{Construction of curious minimal uniquely ergodic homeomorphisms on manifolds: the Denjoy-Rees technique}.
	Annales Scientifiques de lÕ\'{E}cole Normale Sup\'{e}rieure {\bf 40}, no. 2, 2007, 251--308.
\bibitem{Bonk-Schramm}
	Mario Bonk, Oded Schramm
   \emph{Embeddings of Gromov hyperbolic spaces},
   Geom. Funct. Anal., {\bf 10} (2000) 266--306
\bibitem{Bridson-Hafliger}
	Martin Bridson and Andr\'{e} H\"{a}fliger,
	Metric Spaces of Non-positive Curvature, Grundlehren der Mathematischen Wissenschaften 319 (1999)
\bibitem{Calegari-blog}
	Danny Calegari,
	\emph{Big mapping class groups and dynamics},
	\texttt{https://lamington.wordpress.com/2009/06/22/big-mapping-class-groups-and-dynamics/}
\bibitem{Funar-K-S}
	Louis Funar, Christophe Kapoudjian, and Vlad Sergiescu,
	\emph{Asymptotically rigid mapping class groups and Thompson groups},
    Handbook of Teichmuller theory vol III (Editor A.Papadopoulos), 2012, 595-664.
\bibitem{Hensel-Przytycki-Webb}
	Sebastian Hensel, Piotr Przytycki and Richard Webb,
	\emph{1-slim triangles and uniform hyperbolicity for arc graphs and curve graphs}, J. Eur. Math. Soc. (JEMS), {\bf{17}} (2015) 755--762
\bibitem{Kapovich}
	Ilya Kapovich and Nadia Benakli,
	Boundaries of hyperbolic groups. \emph{Combinatorial and geometric group theory}, Contemp. Math., 296, Amer. Math. Soc., Providence, RI, 2002.
\bibitem{Klarreich}
	Erica Klarreich,
	\emph{The boundary at infinity of the curve complex and the relative Teichm\"uller Space},
	preprint
\bibitem{Le Roux}
	Fr\'ed\'eric Le Roux,
	\emph{\'Etude topologique de l'espace des hom\'eomorphismes de Brouwer}, Topology 40(5):1051--1087 (2001).	
\bibitem{Moise}
	Edwin E. Moise,
	Geometric topology in dimensions {$2$} and {$3$}, Graduate Texts in Mathematics, Vol. 47, Springer-Verlag, New York-Heidelberg (1977)
\bibitem{Pho-On}
	Witsarut Pho-on, \emph{Infinite Unicorn Paths and Gromov Boundaries}, arXiv:1508.02296
\end{thebibliography}
\end{document}